\documentclass[12pt]{amsart}
\usepackage{amssymb,amsmath,amsthm}  
\usepackage{graphicx}
\usepackage{bm}
\usepackage{enumerate}
\usepackage{braket}
\usepackage{amsaddr}
\usepackage{amscd}
\usepackage{colordvi}
\usepackage[mathscr]{eucal}
\usepackage{epstopdf}
\usepackage{hyperref}
\usepackage{url}
\usepackage{subcaption}
\usepackage{cleveref}
\numberwithin{equation}{section}
\newtheorem{theorem}{Theorem}[section] 

\newtheorem{lemma}[theorem]{Lemma}

\newtheorem{assumption}[theorem]{Assumption}

\newcommand\norm[1]{\left\lVert#1\right\rVert}
\newcommand{\tu}{\textup}

\newcommand{\dd}{\displaystyle}
\newcommand{\bfa}[1]{\boldsymbol{#1}} 			%

\newcommand{\e}{\epsilon}
\usepackage[numbers]{natbib}
\usepackage{filecontents}
\usepackage[latin9]{inputenc}
\usepackage[letterpaper]{geometry}
\usepackage{pifont}
\usepackage{float}
\usepackage{color}
\usepackage{adjustbox}
\usepackage{diagbox}
\usepackage{multirow}
\usepackage[makeroom]{cancel}
\usepackage{array}
\usepackage{mathrsfs}
\usepackage{tikz}
\usepackage{comment}
\usepackage{csquotes}

\graphicspath{{./Figures/}}

\usepackage{amsfonts}

\usepackage{color}
\usepackage{adjustbox}
\usepackage{diagbox}
\definecolor{black}{rgb}{0,0,0}

\definecolor{red}{rgb}{1,0,0}

\definecolor{blue}{rgb}{0,0,1}

\numberwithin{equation}{section}

\renewcommand{\div}{\mathop{\rm div}\nolimits}

\newcommand{\ep}{\epsilon}

\newcommand{\todo}[1]{{\color{red}{#1}}}

\newcommand{\dx}{ \mathrm{d}x}
\newcommand{\dy}{ \mathrm{d}y}
\newcommand{\dt}{ \mathrm{d}t}

\newcommand{\beq}{\begin{equation}}
\newcommand{\eeq}{\end{equation}}
\newcommand{\beqq}{\begin{equation*}}
\newcommand{\eeqq}{\end{equation*}}
\newcommand{\beqas}{\begin{eqnarray*}}
\newcommand{\eeqas}{\end{eqnarray*}}
\newcommand{\bsp}{\begin{split}}
\newcommand{\eesp}{\end{split}}

\makeatother

\usepackage[english]{babel}

\author[P\MakeLowercase{ark}, C\MakeLowercase{heung},  M\MakeLowercase{ai}, H\MakeLowercase{oang}]
{J\MakeLowercase{un} S\MakeLowercase{ur} R\MakeLowercase{ichard} P\MakeLowercase{ark}, 
S\MakeLowercase{iu} W\MakeLowercase{un} C\MakeLowercase{heung}, 
 T\MakeLowercase{ina} M\MakeLowercase{ai}$^*$,
 V\MakeLowercase{iet}  H\MakeLowercase{a} H\MakeLowercase{oang}}
\date{\today}

\title[M\MakeLowercase{ultiscale simulations for upscaled multi-continuum flows}]
{\textsf{\LARGE M\MakeLowercase{ultiscale simulations for upscaled multi-continuum flows}}}

\begin{document}

\begin{abstract}

We consider in this paper a challenging problem of simulating fluid flows, in complex multiscale media possessing multi-continuum background.  
As an effort to handle this obstacle, model reduction is employed.  In \cite{rh2}, homogenization was nicely applied, to find effective coefficients and homogenized equations (for fluid flow pressures) of a dual-continuum system, with new convection terms and negative interaction coefficients.  However, some degree of multiscale still remains.  This motivates us to propose the generalized multiscale finite element method (GMsFEM), which is coupled with the dual-continuum homogenized equations, toward speeding up the simulation, improving the accuracy as well as clearly representing the interactions between the dual continua.  In our paper, globally, each continuum is viewed as a system and connected to the other throughout the domain.  
We take into consideration the flow transfers between the dual continua and within each continuum itself.  Such multiscale flow dynamics are modeled by the GMsFEM, which systematically generates either uncoupled or coupled multiscale basis (to carry the local characteristics to the global ones), via establishing local snapshots and spectral decomposition in the snapshot space.  As a result, we will work with a system of two equations coupled with some interaction terms, and each equation describes one of the dual continua on the fine grid.  Convergence analysis of the proposed GMsFEM is accompanied with the numerical results, which support the favorable outcomes.

\end{abstract}

\maketitle


\vskip10pt


\vfill

\noindent Jun Sur Richard Park $\cdot$ Siu Wun Cheung

\noindent Department of Mathematics, Texas A\&M University, College Station, TX 77843, USA

\noindent E-mail: pjss1223@math.tamu.edu (Park); tonycsw2905@math.tamu.edu (Cheung) \\

\noindent Tina Mai\textsuperscript{*} (corresponding author)

\noindent Institute of Research and Development, 
Duy Tan University, Da Nang 550000, Viet Nam

\noindent E-mail: maitina@duytan.edu.vn (Mai)\\

\noindent Viet Ha Hoang

\noindent Division of Mathematical Sciences, School of Physical and Mathematical Sciences\\
Nanyang Technological University, Singapore 637371 

\noindent E-mail: vhhoang@ntu.edu.sg (Hoang)


\newpage

\tableofcontents

\section{Introduction}\label{intro}


Fluid flow simulation was early known to be based on the concept of porous medium as a single continuum.  However, 
in nature, a porous medium (as stratum or fissured rock) may possess some degree of fracturing.  This hence motivated the notion of dual continua, or more generally, multicontinua (see \cite{baren}, for instance),
thanks to mean characteristics (porosity, permeability, pressure, ...)\ of the
media and flow.  For example (see \cite{baren}), a dual-continuum background can consist of a matrix (first continuum) and a system of naturally connected fractures (second continuum).  
In such heterogeneous media, the simulation of flow is hard, mainly because of the distinct properties of continua, multiple scales and high contrast. 
In addition, mass would transport among continua and different scales in various forms.  


To handle those difficulties in multi-continuum flow modeling, a straightforward approach is using fine-grid simulation, in several steps.  First, a locally fine grid is established.  Then, the flow equations are discretized on that fine grid, and a  global solution is derived from the set of local solutions.  This approach can be carried out under well-known frameworks, such as the Finite Element Method (FEM) in \cite{fem} and Finite Volume Method (FVM).  


Nevertheless, due to the intricate heterogeneity of the media, especially, multiple scales and high contrast, some type of model reduction is needed for flow simulation.  Common methods involve partitioning the domain of interest into coarse-scale grid blocks, where effective properties in each coarse block are calculated (\cite{homod}).  This computation (in standard upscaling methods based on homogenization) utilizes the fine-scale solutions of local problems in each coarse block or representative volume.  Such a scheme, however, may not reflex multiple crucial modes in each coarse block (including the interaction of continua).  


That resulted in the multi-continuum strategies (\cite{baren, arbogast}) on coarse grid.  Physically, the flow between different continua is described by considering each continuum as a system over the whole domain.  
In fine grid, different continua are adjacent.  In coarse grid, they co-present (via mean characteristics \cite{baren}) at every point of the domain and interact with each other.  Mathematically, a number of equations are established for each coarse block, and each equation represents one of the continua on the fine grid.  For example, in fractured reservoir, the flow equations for the matrix and the system of fractures are written
separately with some interaction terms.  Those interaction terms are coupled (based on the mass conservation law), leading to a system of coupled
equations.  For this purpose, even when each continuum is not topologically connected, we assume that it is connected to the other (throughout the domain and the type of the coupling), provided that it has solely global (not local) effects. 


In these settings, we now discuss the dual-continuum background in our paper.  The first dual porosity model was introduced by Barenblatt, for modeling flow through naturally fissured rock \cite{baren}.  In his work,
two continua were suggested to delineate high and low porosity continua, that is, matrix and system of connected fractures, respectively.  An example about some early work on dual continua based on \cite{baren} is \cite{arbogast} (1990), where homogenization theory was applied, to obtain a general form of the double porosity model of single phase flow, within a naturally fractured reservoir.  Both intra and inter flow transports are modeled for each continuum.  In this paper, the dual-continuum background is in any general form, where the above strategies can be applied.



To overcome the limits of homogenization technique as well as to integrate the heterogeneity of the multicontinua, and to reduce the computational cost, based on the multiscale finite element method (MsFEM) as in \cite{Ms, Msnon}, the generalized multiscale finite element method (GMsFEM) was developed (\cite{G1}).  This method allows one to systematically construct multiple multiscale basis, by adding new degrees of freedom (basis functions) in each coarse block.  These new basis functions are calculated by constructing the local snapshots and performing
local spectral decomposition in the snapshot space.  That is, the producing eigenfunctions can convey the local characteristics to the global ones, via the multiscale basis functions in coarse grid.  


The GMsFEM has been successfully applied to a number of multi-contimuum problems.  Recent example is about shale gas transport in dual-continuum background consisting of organic and inorganic materials (\cite{organic17}).  In this spirit, a third continuum can be added to dual continua as an extension (see \cite{Tony11}, for instance).  More generally, flow simulation in heterogeneously varying multicontinua was investigated (see \cite{mcontinua17, ericmultirichard19a, ericmultiporoelastic19a}, for instance).  

Additionally, there are various and active studies on new model reduction techniques and related numerical methods for multi-continuum systems. They include constraint energy minimizing (CEM) GMsFEM (\cite{cheung2018constraint}) and non-local multi-continuum method (NLMC)(\cite{vasilyeva2019nonlocal,chung2018non,vasilyeva2018three}). These approaches also effectively handle high-contrast as well as multiscale features in multi-continuum media.   



Herein, we develop the GMsFEM for an upscaled multi-continuum system.  That is, as a special case, which we are considering in this paper, multicontinua can occur at many scales.  The big picture is that starting from microscopic scale, the multicontinua are upscaled via homogenization, to reach intermediate scale.  At this stage, the multicontinua still possess some degree of multiscale.  Hence, they are then simulated by the GMsFEM, to arrive at coarse-grid (macroscopic) level.

More specifically, in \cite{park2019hierarchical, rh2} by Park and Hoang, homogenization of multi-continuum systems has been investigated.  Especially, in \cite{rh2}, homogenization has been developed for a two-scale dual-continuum system (for fluid flow pressures), which is new because the given coupled
interaction terms are not uniformly positive and scaled as $\mathcal{O} (1/ \epsilon)$, which is the inverse of the micro-scale $\e$.  The arising homogenized equations still have some grade of multiscale, which motivates our further study on numerical multiscale simulation, for a coupled dual-continuum system with new convection terms and negative interaction coefficients.

The multiscale technique we use to upscale that resulted dual-continuum system is the GMsFEM.  The novelty in our paper is the fine-grid scale, which is the intermediate scale resulting from that homogenization step, so it is different from the fine-grid scale of the GMsFEM in \cite{mcontinua17}.  Here, we derive a combination of the offline GMsFEM and the multi-continuum approaches.  

The convergence analysis is presented for two cases: uncoupled and coupled multiscale basis of the GMsFEM.  For each case, we compare the reference weak solution with the presented coarse-grid approximation (multiscale solution).  In the first case (called uncoupled GMsFEM), multiscale basis functions will be constructed for each continuum separately, by considering only the permeability and disregarding the transfer functions.  Then, we apply the GMsFEM described above.  In the second case (called coupled GMsFEM), multiscale basis functions will be constructed by solving a coupled problem for snapshot space and carrying out a spectral decomposition. From this step, GMsFEM is also utilized.   

In numerical simulations, we focus on coupling the GMsFEM with the multi-continuum approach.  
The reference fine-scale solution is compared with the multiscale solution.  Our numerical results (after using both uncoupled and coupled multiscale basis functions) show that the GMsFEM is able to combine with the multi-continuum approach and gives solution with good accuracy (that is even better with coupled multiscale basis) using few basis functions.  Also, our numerical results demonstrate that the solution obtained via the coupled GMsFEM is more accurate than the one obtained via the FEM.  

The organization of the paper is as follows.  In Section \ref{fus}, we introduce function spaces.  Section \ref{pform} is about problem formulation, where we show the existence and uniqueness of weak solution, and provide fine-scale finite element discretization.  In Section \ref{view}, an overview of the GMsFEM is given, to introduce coarse and fine grids as well as uncoupled and coupled GMsFEM.  Section \ref{ca} is devoted to convergence analysis, for both uncoupled and coupled GMsFEM.  In Section \ref{numer}, numerical results are presented.  Conclusions are summed up in Section \ref{conclude}.




\section{Function spaces}\label{fus}
Let $\Omega$ be our computational domain in $\mathbb{R}^2$.  The spaces of functions, vector fields in $\mathbb{R}^2$, and $2 \times 2$ matrix fields defined over $\Omega$ are respectively denoted by italic capitals (e.g., $L^2(\Omega)$), 
boldface Roman capitals (e.g., $\bfa{V}$), 
and special Roman capitals (e.g., $\mathbb{S}$).  

Consider the space $V: = H_0^1(\Omega) = W_0^{1,2}(\Omega)$.  Its dual space 
(also called the adjoint space), which consists of continuous linear functionals on $H_0^1(\Omega)$, is denoted by 
$H^{-1}(\Omega)$, and 
the value of a functional $f \in H^{-1}(\Omega)$ at a point 
$v \in H_0^1(\Omega)$ is 
denoted by the inner product $\langle f,v \rangle $.  

The Sobolev norm $\| \cdot \|_{W_0^{1,2}(\Omega)}$ is of the form 
\[\|v\|_{W_0^{1,2}(\Omega)} = \left(\|v\|^2_{L^2(\Omega)} + 
\|\nabla v \|^2_{\bfa{L}^2(\Omega)}\right)^{\frac{1}{2}}\,.\]
Here, $\| \nabla v \|_{\bfa{L}^2(\Omega)}:= \| | \nabla v | \|_{\bfa{L}^2(\Omega)}\,,$ where 
$| \nabla v|$ denotes the Euclidean norm of the $2$-component vector-valued function 
$ \nabla v$; and for $\bfa{v} = (v_1,v_2)$, $\| \nabla \bfa{v}\|_{\mathbb{L}^2(\Omega)}:= \| | \nabla \bfa{v}| \|_{\mathbb{L}^2(\Omega)}\,,$ where 
$| \nabla \bfa{v}|$ denotes the Frobenius norm of the $2 \times 2$ matrix $\nabla \bfa{v}$.  
We recall that the Frobenius norm on $\mathbb{L}^2(\Omega)$ is defined by 
$| \bfa{X} |^2 : = \bfa{X} \cdot \bfa{X} = \text{tr}(\bfa{X}^{\text{T}} \bfa{X})\,.$

The dual norm to $\| \cdot \|_{H_0^1(\Omega)}$ is $\| \cdot \|_{H^{-1}(\Omega)}$, i.e.,
\[\| f \|_{H^{-1}(\Omega)} = \sup_{v \in H_0^1(\Omega)}
\frac{|\langle f,v \rangle|}{\|v\|_{H_0^1(\Omega)}}\,.\]
 

For every $1 \leq r < 
\infty$, we use $\bfa{L}^r(0,T;\bfa{X})$ for the Bochner space with the norm 
\[\|\bfa{w}\|_{\bfa{L}^r(0,T;\bfa{X})} := \left(\int_0^T \|\bfa{w} \|_{\bfa{X}}^r \, \dt\right)^{1/r} < + \infty\,, 
\]
\[\|\bfa{w}\|_{\bfa{L}^{\infty}(0,T;\bfa{X})}: = \sup_{0 \leq t \leq T} \|\bfa{w}\|_{\bfa{X}}  < + \infty\,,\]
where $(\bfa{X}, \| \cdot \|_{\bfa{X}})$ is a Banach space.  Also, we define 
\[\bfa{H}^1(0,T; \bfa{X}):= \left \{ \bfa{v} \in \bfa{L}^2(0,T;\bfa{X}) \, : \, \partial_t \bfa{v} \in \bfa{L}^2(0,T;\bfa{X}) \right \}\,.\]

To shorten notation, we denote the space for $\bfa{u}(\cdot,t) = (u_1(\cdot,t), u_2(\cdot,t))$ by $\bfa{V} = V \times V = H^1_0(\Omega) \times H^1_0(\Omega)$, where $t \in [0,T], T > 0$.  

\section{Problem formulation}\label{pform}
In \cite{park2019hierarchical, rh2}, Park and Hoang have studied homogenization of multi-continuum systems (see \cite{baren,warren1963behavior,kazemi1976numerical,wu1988multiple,pruess1982fluid,mcontinua17}, for instance).  Specially, in \cite{rh2}, homogenization was developed for a two-scale dual-continuum system
\begin{equation}
\label{eq:main1_jrp}
 \begin{split}
{\mathcal C}_{11}^\epsilon(\bfa{x}) \dfrac{\partial u_1^\epsilon(\bfa{x},t)}{\partial t}=\text{div}(\kappa_1^\epsilon(\bfa{x})\nabla u_1^\epsilon(\bfa{x},t)) + \frac{1}{\e}Q^\epsilon(\bfa{x})(u_2^\epsilon(\bfa{x},t)-u_1^\epsilon(\bfa{x},t)) + f_1\,,\\
{\mathcal C}_{22}^\epsilon(\bfa{x}) \dfrac{\partial u_2^\epsilon(\bfa{x},t)}{\partial t}=\text{div}(\kappa_2^\epsilon(\bfa{x}) \nabla u_2^\epsilon(\bfa{x},t)) + \frac{1}{\e}Q^\epsilon(\bfa{x})(u_1^\epsilon(\bfa{x},t)-u_2^\epsilon(\bfa{x},t)) + f_2\,,
 \end{split}
 \end{equation} 
 where $\bfa{x} \in \Omega \subset \mathbb{R}^2$, $f_1, f_2 \in L^2(\Omega)$,  $\ep$ represents the microscopic scale of the local variation, and the interaction terms are scaled as $\mathcal{O}(\e^{-1})$ (see \cite{baren,warren1963behavior,kazemi1976numerical,wu1988multiple,pruess1982fluid,mcontinua17}, for instance). 
 Let $Y$ be a unit cube in $\mathbb{R}^2$.
The coefficients ${\mathcal C}_{ii}^\epsilon$, $\kappa_i^\epsilon$ and $Q^\epsilon$ are defined as
\begin{equation}
\label{eq:coeff}
 \begin{split}
{\mathcal C}_{ii}^\epsilon(\bfa{x}) = {\mathcal C}_{ii}\left(\bfa{x},\frac{\bfa{x}}{\epsilon}\right),\ \kappa_i^\epsilon(\bfa{x}) = \kappa_i \left(\bfa{x},\frac{\bfa{x}}{\epsilon}\right) \text{ and } 
\ Q^\epsilon(\bfa{x}) = Q\left(\bfa{x},\frac{\bfa{x}}{\epsilon}\right), \enspace \ \ i = 1,2,
 \end{split}
 \end{equation}
where ${\mathcal C}_{ii}(\bfa{x},\bfa{y})$, $\kappa_i(\bfa{x},\bfa{y})$ and $Q(\bfa{x},\bfa{y})$ are $Y$-periodic functions from $\Omega \times Y$.  The following homogenized equations of the system (\ref{eq:main1_jrp}) were derived in \cite{rh2}:
\begin{equation}
\label{eq:main_hom_jrp}
\begin{split}
\int_Y {\mathcal C}_{11} \, \dy \dfrac{\partial u_{1,0}}{\partial t}
&= \div (\kappa_1^*\nabla u_{1,0}) 
+\div \left[\left(\int_Y \kappa_1 \nabla_y M_1 \, \dy\right)(u_{2,0}-u_{1,0})\right]
\\ &+ \sum_{i=1}^2 \left[\left(\int_Y QN^i_2 \, \dy\right)\frac{\partial u_{2,0}}{\partial x_i} - \left(\int_Y QN^i_1 \, \dy\right)\frac{\partial u_{1,0}}{\partial x_i}\right] \\
&- \left(\int_YQ(M_1+M_2) \, \dy\right)(u_{2,0}-u_{1,0}) + f_1\,,\\
\int_Y{\mathcal C}_{22} \dy \dfrac{\partial u_{2,0}}{\partial t} 
&= \div (\kappa_2^*\nabla u_{2,0}) 
+\div \left[\left(\int_Y \kappa_2 \nabla_y M_2 \, \dy\right)(u_{1,0}-u_{2,0})\right]
\\& + \sum_{i=1}^2 \left[\left(\int_Y QN^i_1 \, \dy\right)\frac{\partial u_{1,0}}{\partial x_i} - \left(\int_Y QN^i_2 \, \dy\right)\frac{\partial u_{2,0}}{\partial x_i}\right] \\
&- \left(\int_YQ(M_1+M_2) \, \dy\right)(u_{1,0}-u_{2,0})+ f_2\,,
\end{split}
\end{equation}

\noindent where $\kappa_1^*$ and $\kappa_2^*$ are symmetric and positive definite, $f_1$ and $f_2$ are in $L^2(\Omega)$. The coefficients $\displaystyle \int_Y \kappa_i \nabla_y M_i \, \dy$ and $\displaystyle \int_Y QN^i_{j} \, \dy$ (where $i,j= 1,2$) can be either positive or negative, and $\left(-\displaystyle \int_YQ(M_1+M_2) \, \dy\right)$ is uniformly negative in $\Omega$.
These homogenized equations still possess some degree of multiscale.  This motivates our research (herein) on numerical multiscale simulation for a dual-continuum system with general convection and reaction terms:
\beq
\label{eq:main1*}
\begin{split}
& {\mathcal C}_{11}(\bfa{x}) \frac{\partial u_{1}(\bfa{x},t)}{\partial t} -\div (\kappa_1 (\bfa{x}) \nabla u_1(\bfa{x},t)) + \bfa{b}_1(\bfa{x})  \cdot \nabla(u_1(\bfa{x},t)-u_2(\bfa{x},t)) \\
& \; + \; Q_1(\bfa{x}) (u_1(\bfa{x},t)-u_2(\bfa{x},t)) = f_1(\bfa{x})\,,\\
& {\mathcal C}_{22}(\bfa{x}) \frac{\partial u_{2}(\bfa{x},t)}{\partial t} -\div (\kappa_2 (\bfa{x}) \nabla u_2(\bfa{x},t)) + \bfa{b}_2(\bfa{x})  \cdot \nabla(u_2(\bfa{x},t)-u_1(\bfa{x},t)) \\
& \; + \; Q_2(\bfa{x}) (u_2(\bfa{x},t)-u_1(\bfa{x},t)) = f_2(\bfa{x})\,,
\end{split}
\eeq
in $\Omega \times (0,T)$, with the Dirichlet boundary condition $u_1(\bfa{x}) = u_2(\bfa{x}) = 0$ on $\partial \Omega\times (0,T)$, and with suitable initial conditions (when $t = 0, T$), given $\bfa{f}(\bfa{x}) = (f_1(\bfa{x}), f_2(\bfa{x})) \in \bfa{L}^{2}(\Omega)$. We will show later that (\ref{eq:main1*}) has a unique solution under certain conditions.  One of the main difficulties as well as contributions of our paper is that in (\ref{eq:main1*}), we use different $Q_1$ and $Q_2$ rather than the same $Q$ in (\ref{eq:main_hom_jrp}).

The variational form of (\ref{eq:main1*}) is as follows:  Find $\bfa{u} = (u_1,u_2) \in \bfa{V}$ such that
\beq
\label{eq:vartt}
\begin{split}
& \int_{\Omega}{\mathcal C}_{11} \frac{\partial u_{1}}{\partial t}\phi_1 \, \dx
+\int_{\Omega} \kappa_1(\bfa{x})  \nabla u_{1} \cdot \nabla \phi_1 \, \dx
+ \int_\Omega \bfa{b}_1(\bfa{x})  \cdot \nabla(u_{1}-u_{2}) \phi_1\, \dx\\
& \; +\int_{\Omega} Q_1(\bfa{x}) (u_{1}-u_{2})\phi_1  \, \dx 
 = \int_{\Omega} f_1 \phi_1 \, \dx \,,\\
 & \int_{\Omega}{\mathcal C}_{22} \frac{\partial u_{2}}{\partial t}\phi_2 \, \dx
+\int_{\Omega} \kappa_2(\bfa{x}) \nabla u_{2}\cdot \nabla \phi_2 \, \dx
 + \int_{\Omega} \bfa{b}_2(\bfa{x})  \cdot \nabla(u_{2}-u_{1})\phi_2  \, \dx\\
 & \; + \int_{\Omega} Q_2(\bfa{x}) (u_{2}-u_{1}) \phi_2  \, \dx 
 =\int_{\Omega} f_2 \phi_2 \, \dx\,,
\end{split}
\eeq
for all $\bfa{\phi} = (\phi_1, \phi_2) \in \bfa{V}$, for a.e.\ $t \in (0,T)$.
Before studying this problem, we first consider the following interesting static dual-continuum system:
\beq
\label{eq:main1}
\begin{split}
&-\div (\kappa_1(\bfa{x}) \nabla u_1 (\bfa{x})) + \bfa{b}_1(\bfa{x})  \cdot \nabla(u_1(\bfa{x})-u_2(\bfa{x})) + Q_1(\bfa{x}) (u_1(\bfa{x})-u_2(\bfa{x})) = f_1(\bfa{x})\,,\\
&-\div (\kappa_2 (\bfa{x}) \nabla u_2(\bfa{x})) + \bfa{b}_2(\bfa{x})  \cdot \nabla(u_2(\bfa{x})-u_1(\bfa{x})) + Q_2(\bfa{x}) (u_2(\bfa{x})-u_1(\bfa{x})) = f_2(\bfa{x})\,,
\end{split}
\eeq
in $\Omega$, with the Dirichlet boundary condition $u_1(\bfa{x}) = u_2(\bfa{x}) = 0$ on $\partial \Omega$, where
$\kappa_1(\bfa{x})$ and $\kappa_2(\bfa{x})$ are permeability coefficients in high contrast media, provided $\bfa{f}(\bfa{x}) = (f_1(\bfa{x}), f_2(\bfa{x})) \in \bfa{L}^{2}(\Omega)$. 

For later use, we define 
\begin{equation}\label{syasy}
 \bfa{b}_s = \dfrac{\bfa{b}_1+\bfa{b}_2}{2}\,,   \quad \bfa{b}_a = \dfrac{\bfa{b}_1-\bfa{b}_2}{2}\,, \quad Q_s = \dfrac{Q_1 + Q_2}{2}\,, \quad Q_a = \dfrac{Q_1 - Q_2}{2}\,,
\end{equation}
in variable $\bfa{x}\,.$

Throughout this section, we assume the following. 
\begin{assumption}
\label{bdd}
There are some positive constants $\bar{\mathcal{C}}, \underline{\mathcal{C}}$, $\bar{ b}$, $\bar{Q}$ and $\bar{\kappa}, {\underline{\kappa}}$ such that $\bar{\mathcal{C}} \geq \mathcal{C}_{ii}\geq \underline{\mathcal{C}}$, $|\bfa{b}_i | \leq \bar{b} $, $ |Q_{i}| \leq \bar{Q}$, $ \bar{\kappa} \geq \kappa_i \geq \underline{\kappa}$ ($i=1,2$), and we further assume that 
$1 > \bar{ b} / \sqrt{\underline{\kappa}}\,,$  $|\bfa{b}_s| \gg |\bfa{b}_a|$ and $|Q_s| \gg |Q_a|\,.$
\end{assumption}
The system \ref{eq:main1} can be written in the variational form
\beq\label{eq:var}
\begin{split}
& \int_{\Omega} \kappa_1 (\bfa{x}) \nabla u_1(\bfa{x}) \cdot \nabla \phi_1(\bfa{x}) \, \dx
+ \int \bfa{b}_1  \cdot \nabla(u_1(\bfa{x})-u_2(\bfa{x})) \phi_1(\bfa{x}) \, \dx \\
 & \; + \; \int_{\Omega} Q_1(\bfa{x}) (u_1(\bfa{x})-u_2(\bfa{x}))\phi_1(\bfa{x})  \, \dx 
 = \int_{\Omega} f_1(\bfa{x}) \phi_1(\bfa{x}) \, \dx\,,\\
& \int_{\Omega} \kappa_2 (\bfa{x}) \nabla u_2(\bfa{x})\cdot \nabla \phi_2(\bfa{x}) \, \dx
 + \int_{\Omega} \bfa{b}_2  \cdot \nabla(u_2(\bfa{x})-u_1(\bfa{x}))\phi_2(\bfa{x})  \, \dx \\
& \; + \; \int_{\Omega} Q_2(\bfa{x}) (u_2(\bfa{x})-u_1(\bfa{x})) \phi_2(\bfa{x})  \, \dx 
 =\int_{\Omega} f_2(\bfa{x}) \phi_2(\bfa{x}) \, \dx\,,
\end{split}
\eeq
for all $\phi_1(\bfa{x}),\ \phi_2(\bfa{x}) \in V $.  We define a norm $||\cdot||_a$ on the space $\bfa{V}$ as  
\beq\label{abi}
||(u_1,u_2)||_a = \left(\left \|\kappa^{\frac{1}{2}}_1\nabla u_1\right\|_{L^2(\Omega)}^2+\left \|\kappa^{\frac{1}{2}}_2\nabla u_2 \right \|_{L^2(\Omega)}^2\right)^{\frac{1}{2}}\,.
\eeq
We define a bilinear form $b(\cdot, \cdot): \bfa{V} \times \bfa{V} \longrightarrow \mathbb{R} $ as
\beq\label{bbi}
\begin{split}
b((u_1, u_2 ),(v_1, v_2)) 
&= \int_{\Omega} \kappa_1 \nabla u_1 \cdot \nabla v_1 \, \dx + \int_{\Omega} \kappa_2 \nabla u_2 \cdot\nabla v_2 \, \dx \\
& \; +\int_{\Omega}  \bfa{b}_1 \cdot \nabla (u_1-u_2) v_1 \, \dx
+\int_{\Omega}  \bfa{b}_2 \cdot \nabla (u_2-u_1) v_2 \, \dx\\
& \; +\int_{\Omega}  Q_1 (u_1-u_2) v_1 \, \dx +\int_{\Omega}  Q_2  (u_2-u_1) v_2 \, \dx \,.
\end{split}
\eeq
\subsection{Existence and uniqueness of weak solutions}
In this section, we will show that each of the systems (\ref{eq:var}) and (\ref{eq:vartt}) has a unique solution under certain conditions. 
\begin{lemma}
\label{bddgarding}
Under Assumption \ref{bdd}, there are some positive constants $K$, $\alpha$ and $C_b$ such that for all $\bfa{u} = (u_1,u_2), \, \bfa{v} = (v_1,v_2) 
\in \bfa{V}$, we have 
\begin{align}
b((u_1,u_2),(v_1,v_2)) & \leq C_b ||\bfa{u}||_a \ ||\bfa{v}||_a \,,  \label{boundb} \\
b((u_1,u_2),(u_1,u_2)) + K ||\bfa{u}||_{\bfa{L}^2(\Omega)}^2 & \geq \alpha ||\bfa{u}||_a^2  \,. \label{gardingb}
\end{align}
\end{lemma}
\begin{proof}
First, we prove (\ref{boundb}).  Note that
\beq
\begin{split}
b((u_1, u_2 ),(v_1, v_2)) & \leq \sum\limits_{i=1}^2 \left \|\kappa^{\frac{1}{2}}_i \nabla u_i \right \|_{L^2(\Omega)} \ \left \| \kappa^{\frac{1}{2}}_i \nabla v_i \right \|_{L^2(\Omega)}\\
& \; +\frac{\bar{b}}{\sqrt{\underline{\kappa}}} \displaystyle \sum_{i=1}^2 \sum_{j=1}^2 \left \|\kappa^{\frac{1}{2}}_i \nabla u_i \right \|_{L^2(\Omega)} 
|| v_j ||_{L^2(\Omega)}
 +\bar{Q} \displaystyle \sum_{i=1}^2 \sum_{j=1}^2 || u_i ||_{L^2(\Omega)} || v_j ||_{L^2(\Omega)}\,.
\end{split}
\eeq
By the Poincar\'{e} inequality, there exits a positive constant $C_p(\Omega)$ such that
\begin{equation}\label{poincare}
  ||v_i||_{L^2(\Omega)} \leq C_p ||\nabla v_i||_{L^2(\Omega)} \leq \dfrac{C_p}{\sqrt{\underline{\kappa}}} \left \| \kappa_i^{\frac{1}{2}} \nabla v_i \right\|_{L^2(\Omega)}\,,
  \end{equation}
  for all $v_i \in H^1_0(\Omega), \ i=1,2\,.$ Thus, we obtain 
\beq\label{bb}
\begin{split}
 b((u_1, u_2 ),(v_1, v_2)) 
& \leq \sum\limits_{i=1}^2 \left \|\kappa^{\frac{1}{2}}_i \nabla u_i \right\|_{L^2(\Omega)} \ \left \| \kappa^{\frac{1}{2}}_i \nabla v_i \right\|_{L^2(\Omega)}\\
 & \; +\frac{\bar{b}C_p}{\underline{\kappa}} \displaystyle \sum_{i=1}^2 \sum_{j=1}^2 \left \| \kappa^{\frac{1}{2}}_i \nabla u_i \right\|_{L^2(\Omega)} \ 
\left \|\kappa^{\frac{1}{2}}_j \nabla v_j \right\|_{L^2(\Omega)}\\
& \; + \frac{ \bar{Q}{C_p}^2}{\underline{\kappa}}  \displaystyle \sum_{i=1}^2 \sum_{j=1}^2 \left \| \kappa^{\frac{1}{2}}_i \nabla u_i \right\|_{L^2(\Omega)} \ \left \| \kappa^{\frac{1}{2}}_j \nabla v_j \right\|_{L^2(\Omega)}\\
&\leq \bigg( \left(1+\frac{2\bar{b}C_p}{\underline{\kappa}}+\frac{2 \bar{Q}{C_p}^2}{\underline{\kappa}}\right)\displaystyle \sum_{i=1}^2 
 \left \|\kappa^{\frac{1}{2}}_i \nabla u_i \right\|_{L^2(\Omega)}^2 \bigg)^{\frac{1}{2}}\\
 & \; \cdot \bigg( \left(1+\frac{2\bar{b}C_p}{\underline{\kappa}}+\frac{2 \bar{Q}{C_p}^2} {\underline{\kappa}}\right)\displaystyle \sum_{i=1}^2 
 \left \| \kappa^{\frac{1}{2}}_i \nabla v_i \right\|_{L^2(\Omega)}^2 \bigg)^{\frac{1}{2}}\,.
\end{split}
\eeq
From (\ref{bb}), we obtain the boundedness of $b(\cdot,\cdot)$ as in (\ref{boundb}).  

To prove (\ref{gardingb}), we first note that
\beq
\begin{split}
b((u_1, u_2 ),(u_1, u_2)) 
& = \int_{\Omega} \kappa_1 \nabla u_1\cdot \nabla u_1 \, \dx + \int_{\Omega} \kappa_2 \nabla u_2 \cdot \nabla u_2 \, \dx \\
& \; +\int_{\Omega}  \bfa{b}_1 \cdot \nabla (u_1-u_2) u_1 \, \dx
+\int_{\Omega}  \bfa{b}_2 \cdot \nabla (u_2-u_1) u_2 \, \dx \\
& \; +\int_{\Omega}  Q_1  (u_1-u_2) u_1 \, \dx +\int_{\Omega}  Q_2  (u_2-u_1) u_2 \, \dx\\
&  \geq \displaystyle \sum_{i=1}^2 \left \|\kappa^{\frac{1}{2}}_i \nabla u_i \right\|_{L^2(\Omega)}^2
-\frac{\bar{b}}{\sqrt{\underline{\kappa}}} \displaystyle \sum_{i=1}^2 \sum_{j=1}^2 \left\|\kappa^{\frac{1}{2}}_i \nabla u_i \right \|_{L^2(\Omega)} 
|| u_j ||_{L^2(\Omega)}\\
& \; -\bar{Q} \displaystyle \sum_{i=1}^2 \sum_{j=1}^2 || u_i ||_{L^2(\Omega)} || u_j ||_{L^2(\Omega)}\\
& \geq \displaystyle \sum_{i=1}^2 \left \| \kappa^{\frac{1}{2}}_i \nabla u_i \right \|_{L^2(\Omega)}^2
-\frac{\bar{b}}{2\sqrt{\underline{\kappa}}} \displaystyle \sum_{i=1}^2 \sum_{j=1}^2 \left (\left \|\kappa^{\frac{1}{2}}_i \nabla u_i \right\|_{L^2(\Omega)}^2 + 
|| u_j ||_{L^2(\Omega)}^2 \right)\\
& \; -\frac{\bar{Q}}{2} \displaystyle \sum_{i=1}^2 \sum_{j=1}^2 \left( || u_i ||_{L^2(\Omega)}^2 + || u_j ||_{L^2(\Omega)}^2\right)\\
& =\displaystyle \sum_{i=1}^2 \left \| \kappa^{\frac{1}{2}}_i \nabla u_i \right\|_{L^2(\Omega)}^2
-\frac{\bar{b}}{\sqrt{\underline{\kappa}}} \displaystyle \sum_{i=1}^2 \left(\left \| \kappa^{\frac{1}{2}}_i \nabla u_i \right \|_{L^2(\Omega)}^2 + 
|| u_i ||_{L^2(\Omega)}^2\right)\\
& \; - 2\bar{Q} \displaystyle \sum_{i=1}^2  || u_i ||_{L^2(\Omega)}^2 \\
& = \left(1-\frac{\bar{b}}{\sqrt{\underline{\kappa}}}\right) \displaystyle \sum_{i=1}^2 \left \|\kappa^{\frac{1}{2}}_i \nabla u_i \right \|_{L^2(\Omega)}^2
-\left(\frac{\bar{b}}{\sqrt{\underline{\kappa}}} +2\bar{Q} \right) \displaystyle \sum_{i=1}^2  
|| u_i ||_{L^2(\Omega)}^2 \,.
\end{split}
\eeq
Thus, we deduce that
\beq\label{K}
\begin{split}
b((u_1, u_2 ),(u_1, u_2)) + K \displaystyle \sum_{i=1}^2  
|| u_i ||_{L^2(\Omega)}^2 \geq \alpha \displaystyle \sum_{i=1}^2 \left \| \kappa^{\frac{1}{2}}_i \nabla u_i \right \|_{L^2(\Omega)}^2\,,
\end{split}
\eeq
where $K = \dfrac{\bar{b}}{\sqrt{\underline{\kappa}}} +2\bar{Q} $ and
$  1-\dfrac{\bar{b}}{\sqrt{\underline{\kappa}}} \geq \alpha >0$ by Assumption \ref{bdd}.  Hence, (\ref{gardingb}) holds.
\end{proof}


The following assumption is made for later use. 

\begin{assumption}
\label{assm_alpha}
We assume that $\alpha > \dfrac{K \, C_p}{\sqrt{\underline{\kappa}}}\,,$ where $C_p$, $K$ and $\alpha$ are from the proof of Lemma \ref{bddgarding}.
\end{assumption}
We now present the main results of this section under Assumptions \ref{bdd} and \ref{assm_alpha}.

\begin{lemma}\label{coercivel}
Under Assumption \ref{bdd} and \ref{assm_alpha}, we have
\beq
\bsp
b((u_1, u_2 ),(u_1, u_2)) 
\geq C_c ||\bfa{u}||_a^2\,,
\end{split}
\eeq
for some constant $C_c >0$.
\end{lemma}
\begin{proof}
From (\ref{K}) in the proof of Lemma \ref{bddgarding} and the Poincar\'{e}
inequality (\ref{poincare}), we obtain
\beq\label{KnotK0}
\begin{split}
b((u_1, u_2 ),(u_1, u_2)) 
+  \dfrac{K \, C_p}{\sqrt{\underline{\kappa}}} \displaystyle \sum_{i=1}^2  
\left \| \kappa_i^{\frac{1}{2}} \nabla u_i \right \|_{L^2(\Omega)}^2 
\geq \alpha \displaystyle \sum_{i=1}^2 \left \| \kappa^{\frac{1}{2}}_i \nabla u_i \right\|_{L^2(\Omega)}^2\,.
\end{split}
\eeq
Then, it follows that 
\beq\label{coerciveb}
\begin{split}
b((u_1, u_2 ),(u_1, u_2)) 
\geq C_c ||\bfa{u}||_a^2\,,
\end{split}
\eeq
where $C_c = \alpha-\dfrac{K\, C_p}{\sqrt{\underline{\kappa}}} > 0$ by Assumption \ref{assm_alpha}.
\end{proof}

\begin{theorem}
\label{unique_alpha}
Under Assumption \ref{bdd} and \ref{assm_alpha}, we have a unique solution of the problem (\ref{eq:var}) with respect to $||\cdot||_a$.
\end{theorem}
\begin{proof}
The theorem directly results from Lemmas \ref{bddgarding}, \ref{coercivel} and the Lax-Milgram Theorem. 
\end{proof}

Also for later use, note that under Assumption \ref{bdd} and \ref{assm_alpha}, the following assumptions are satisfied. 
\begin{assumption}\label{assumeb}
\label{bdd_coer_a}
There exist constants $C_1$, $C_2$ $>0$ such that
\beq\label{boundcoercive}
\begin{split}
&b((u_1,u_2),(v_1,v_2)) \leq C_1 ||\bfa{u}||_a \ ||\bfa{v}||_a, \\
&b((u_1,u_2),(u_1,u_2))  \geq  C_2 ||\bfa{u}||_a^2\,,
\end{split}
\eeq
for all $\bfa{u} = (u_1,u_2), \, \bfa{v}=(v_1,v_2) \in \bfa{V}$.
\end{assumption}

\begin{theorem}
Under Assumption \ref{bdd}, the problem (\ref{eq:vartt}) has a unique solution. 
\end{theorem}
\begin{proof}
We refer to \cite{wloka,rh2} and Lemma \ref{bddgarding} for the proof. 
\end{proof}
\subsection{Fine-scale finite element discretization}\label{fsfed}
We provide finite element approximation of the solutions to (\ref{eq:var}) and (\ref{eq:vartt}).
Let $\bfa{V}_h = V^1_h\times V^2_h = V_h\times V_h\, (\subset \bfa{V})$, a Cartesian product space, be the first-order Galerkin finite element basis space, with respect to the fine grid $\mathcal{T}_h$.  
That is, in our paper, $V^i_h = V_h$ is a conforming  finite element space of each continuum $i \, (\tu{for } i=1,2)$ on $\mathcal{T}_h$.

We first consider the proposed static case (\ref{eq:main1}), that is, solving the following problem for $\bfa{u}_h = (u_{h,1}, u_{h,2}) \,  (\in \bfa{V}_h)$:
\beq
\label{eq:var2}
\begin{split}
& \int_{\Omega} \kappa_1 (\bfa{x}) \nabla u_{h,1}(\bfa{x}) \cdot \nabla \phi_1(\bfa{x}) \, \dx
+ \int \bfa{b}_1(\bfa{x})  \cdot \nabla(u_{h,1}(\bfa{x})-u_{h,2}(\bfa{x})) \phi_1(\bfa{x}) \, \dx\\
 & \; +\int_{\Omega} Q_1(\bfa{x}) (u_{h,1}(\bfa{x})-u_{h,2}(\bfa{x}))\phi_1(\bfa{x})  \, \dx 
 = \int_{\Omega} f_1(\bfa{x}) \phi_1(\bfa{x}) \, \dx\,,\\
& \int_{\Omega} \kappa_2 (\bfa{x}) \nabla u_{h,2}(\bfa{x})\cdot \nabla \phi_2(\bfa{x}) \, \dx
 + \int_{\Omega} \bfa{b}_2(\bfa{x})  \cdot \nabla(u_{h,2}(\bfa{x})-u_{h,1}(\bfa{x}))\phi_2(\bfa{x})  \, \dx \\
 & \; + \int_{\Omega} Q_2(\bfa{x}) (u_{h,2}(\bfa{x})-u_{h,1}(\bfa{x})) \phi_2(\bfa{x})  \, \dx 
 =\int_{\Omega} f_2(\bfa{x}) \phi_2(\bfa{x}) \, \dx\,,
\end{split}
\eeq
for all $(\phi_1, \phi_2) \in \bfa{V}_h$.
\begin{lemma} 
\label{lemma:1}
Assuming $\bfa{u}  \in \bfa{H}^2(\Omega)$, we have 
\[\displaystyle\inf_{\bfa{v} \in \bfa{V}_h} ||\bfa{u} - \bfa{v}||_a \leq C_A (\bar{\kappa}) h \, \|\bfa{u}\|_{\bfa{H}^2(\Omega)}\,,\] where $\bar{\kappa} \geq \kappa_i$ (as in Assumption \ref{bdd} for $i=1,2$).
\end{lemma}
\begin{proof}
The proof is quite standard by the defintion (\ref{abi}) of norm $\| \cdot \|_a$ and the Bramble-Hilbert Lemma.
\end{proof}
Let $\langle \bfa{u},\bfa{v} \rangle_{\bfa{L}^2(\Omega)}:=\displaystyle \int_\Omega u_1 v_1 \, \dx + \displaystyle \int_\Omega u_2 v_2 \, \dx$,
where $\bfa{u} = (u_1, u_2), \ \bfa{v} = (v_1, v_2) \in \bfa{V}$.
We consider the adjoint problem of (\ref{eq:var}) : Find $\bfa{w} \in \bfa{V}$ that satisfies
\begin{equation}\label{adj2}
b(\bfa{v},\bfa{w}) = \langle \bfa{f}, \bfa{v} \rangle_{\bfa{L}^2(\Omega)}, \quad \text{for all } \bfa{v} \in \bfa{V}\,.
\end{equation}

\begin{theorem}
\label{unique1}
Assume that each of the problem (\ref{eq:var}) and its corresponding adjoint problem has a unique solution in $\bfa{V}$.  We further assume that the solution $\bfa{w} = (w_1, w_2) \in \bfa{V}$ of the above adjoint problem (\ref{adj2}) satisfies
\beq\label{CR}
\|\bfa{w}\|_{\bfa{H}^2(\Omega)} \leq C_R ||\bfa{f}||_{\bfa{L}^2(\Omega)} \,,
\eeq
for all $\bfa{f} = (f_1, f_2) \in \bfa{L}^{2}(\Omega)$.  Let $\bfa{u} \in \bfa{V}$ be the solution of (\ref{eq:var}).  Then, there are positive constants $h_0$ and $C$ such that for all $h \leq h_0$, the problem (\ref{eq:var2}) has a unique solution $\bfa{u}_h = (u_{h,1}, u_{h,2}) \in \bfa{V}_h$ that satisfies
\beq\label{Cba}
||\bfa{u} - \bfa{u}_h||_a \leq C\inf\limits_{\bfa{v} \in \bfa{V}_h} ||\bfa{u}-\bfa{v}||_a\,,
\eeq
where we may take $C = 2C_b/\alpha$, with $C_b$ and $\alpha$ from Lemma \ref{bddgarding}. 
\end{theorem}
\begin{proof}
The Theorem is proved based on the procedure in \cite{brenner2007mathematical}. 
From Lemma \ref{bddgarding}, we get
\beq\label{alpha}
\bsp
\alpha ||\bfa{u}-\bfa{u}_h||_a^2 \leq b(\bfa{u}-\bfa{u}_h, \bfa{u}-\bfa{u}_h) + K  ||\bfa{u} -\bfa{u}_h||_{\bfa{L}^2(\Omega)}^2\,,
\end{split}
\eeq
where $K$ and $\alpha$ are as in the proof of Lemma \ref{bddgarding}.  From (\ref{eq:var2}), for any $\bfa{v} \in \bfa{V}_h$, we always have $b(\bfa{u}-\bfa{u}_h, \bfa{v}) = 0$.  Thus,
\beq
\label{eq:main2}
\bsp
& b(\bfa{u}-\bfa{u}_h, \bfa{u}-\bfa{u}_h) + K ||\bfa{u} -\bfa{u}_h||_{\bfa{L}^2(\Omega)}^2 \\
& = b(\bfa{u}-\bfa{u}_h, \bfa{u}-\bfa{v}) + K ||\bfa{u} -\bfa{u}_h||_{\bfa{L}^2(\Omega)}^2\\
&\leq C_b ||\bfa{u}-\bfa{u}_h ||_a \ ||\bfa{u}-\bfa{v} ||_a+ K ||\bfa{u} -\bfa{u}_h||_{\bfa{L}^2(\Omega)}^2\,,
\end{split}
\eeq
where the last inequality follows from (\ref{boundb}).  Let $\bfa{w} \in \bfa{V}$ be the solution to the problem (\ref{adj2}) 
with $\bfa{f} = \bfa{u}-\bfa{u}_h$, that is,
$b(\bfa{v},\bfa{w}) = \langle \bfa{u}-\bfa{u}_h, \bfa{v} \rangle_{\bfa{L}^2(\Omega)}$ for all $\bfa{v} \in \bfa{V}$.
Then, for any $\bfa{w}_h \in \bfa{V}_h$, we obtain 
\beq\label{L2norm}
\bsp
\| \bfa{u}-\bfa{u}_h \|^2_{\bfa{L}^2(\Omega)}  & = \langle \bfa{u}-\bfa{u}_h,\bfa{u}-\bfa{u}_h\rangle_{\bfa{L}^2(\Omega)}  = b(\bfa{u}-\bfa{u}_h, \bfa{w}) = b(\bfa{u}-\bfa{u}_h, \bfa{w}-\bfa{w}_h )\\
& \leq C_b||\bfa{u}-\bfa{u}_h ||_a \ ||\bfa{w}-\bfa{w}_h ||_a\,.
\end{split}
\eeq
By Lemma \ref{lemma:1} for $||\bfa{w}-\bfa{w}_h ||_a$, (\ref{L2norm}) becomes
\beq\label{3_27}
\bsp
\| \bfa{u}-\bfa{u}_h \|^2_{\bfa{L}^2(\Omega)}  &\leq C_b C_A h ||\bfa{u}-\bfa{u}_h ||_a \ \|\bfa{w}\|_{\bfa{H}^2(\Omega)}\\
&\leq C_b C_A C_R h ||\bfa{u}-\bfa{u}_h ||_a \ ||\bfa{u}-\bfa{u}_h ||_{\bfa{L}^2(\Omega)}\,,
\end{split}
\eeq
where the last inequality follows from assumption (\ref{CR}).  Simplifying (\ref{3_27}), we get
\beq
\bsp
\| \bfa{u}-\bfa{u}_h \|_{\bfa{L}^2(\Omega)}  \leq
C_b C_A C_R h ||\bfa{u}-\bfa{u}_h ||_a\,.
\end{split}
\eeq
From this inequality and (\ref{eq:main2}), we derive from (\ref{alpha}) that
\beq
\label{eq:main3}
\bsp
\alpha ||\bfa{u}-\bfa{u}_h||_a^2 
\leq C_b ||\bfa{u}-\bfa{u}_h ||_a\ ||\bfa{u}-\bfa{v} ||_a+  K (C_b C_A C_R h)^2 ||\bfa{u}-\bfa{u}_h ||_a^2\,.
\end{split}
\eeq
For $h \leq h_0$, where $h_0 = \dfrac{\sqrt{\alpha}}{ \sqrt{2K} C_b C_A C_R}$, we obtain 
\beq
||\bfa{u}-\bfa{u}_h||_a \leq  \dfrac{2C_b}{\alpha} ||\bfa{u}-\bfa{v} ||_a\,,
\eeq
for all $\bfa{v} \in \bfa{V}_h$, and the desired result (\ref{Cba}) follows.  The proof of uniqueness of the solution to (\ref{eq:var2}) is quite straightforward (\cite{brenner2007mathematical}).
\end{proof}
We now investigate the dynamic case, that is, the variational problem (\ref{eq:vartt}) of (\ref{eq:main1*}) for $\bfa{u}_h = (u_{h,1}, u_{h,2}) \in \bfa{V}_h$: 
\beq
\label{eq:var4}
\begin{split}
& \int_{\Omega}{\mathcal C}_{11} \frac{\partial u_{h,1}}{\partial t}\phi_1 \, \dx
+\int_{\Omega} \kappa_1(\bfa{x})  \nabla u_{h,1} \cdot \nabla \phi_1 \, \dx
+ \int_\Omega \bfa{b}_1(\bfa{x})  \cdot \nabla(u_{h,1}-u_{h,2}) \phi_1\, \dx\\
 & \; +\int_{\Omega} Q_1(\bfa{x}) (u_{h,1}-u_{h,2})\phi_1  \, \dx 
 = \int_{\Omega} f_1 \phi_1 \, \dx\,,\\
 & \int_{\Omega}{\mathcal C}_{22} \frac{\partial u_{h,2}}{\partial t}\phi_2 \, \dx
+\int_{\Omega} \kappa_2(\bfa{x}) \nabla u_{h,2}\cdot \nabla \phi_2 \, \dx
 + \int_{\Omega} \bfa{b}_2(\bfa{x})  \cdot \nabla(u_{h,2}-u_{h,1})\phi_2  \, \dx\\
 & \; + \int_{\Omega} Q_2(\bfa{x}) (u_{h,2}-u_{h,1}) \phi_2  \, \dx 
 =\int_{\Omega} f_2 \phi_2 \, \dx\,,
\end{split}
\eeq
for all $(\phi_1, \phi_2) \in \bfa{V}_h$ and a.e.\ $t \in (0,T)$.
We define the following bilinear forms in $\bfa{V}\times \bfa{V}$:
\beq\label{cabi}
\bsp
&c((u_1,u_2),(v_1,v_2))
 = \int_{\Omega}{\mathcal C}_{11} u_1 v_1 \, \dx
 + \int_{\Omega}{\mathcal C}_{22} u_2 v_2 \, \dx\,,\\
&a((u_1,u_2),(v_1,v_2)) = \int_{\Omega} \kappa_1  \nabla u_{1} \cdot \nabla v_1 \, \dx
+\int_{\Omega} \kappa_2 \nabla u_{2}\cdot \nabla v_2 \, \dx\,.\\
\end{split}
\eeq
Let us hence define the norms $||\bfa{u}||_c^2 = c(\bfa{u},\bfa{u}) = \langle \bfa{u}, \bfa{u} \rangle_c$ and $||\bfa{u}||_a^2 = a(\bfa{u},\bfa{u}) = \langle \bfa{u}, \bfa{u}\rangle_a\,.$
\begin{lemma}
\label{ceafem1}
Under Assumption \ref{bdd_coer_a},
we have 
\beq\label{lemmain1}
\bsp
& ||\bfa{u} (\cdot,T)-\bfa{u}_h(\cdot,T)||_c^2 
+ \int_0^T ||\bfa{u}-\bfa{u}_h||_{a}^2 \, \dt\\
& \; \leq
 C  \inf\limits_{\bfa{w} \in \bfa{V}_h} \bigg(\int_0^T \norm{\frac{\partial (\bfa{w} -\bfa{u})}{\partial t}}_c^2 \, \dt 
 + \int_0^T \norm{\bfa{w}-\bfa{u}}_{a}^2 \, \dt
+ ||\bfa{w} (\cdot,0)-\bfa{u}(\cdot,0)||_c^2\bigg)\,,
\end{split}
\eeq
where $\bfa{u}$ and $\bfa{u}_{h}$ satisfy (\ref{eq:vartt}) and (\ref{eq:var4}), respectively.
\end{lemma}
\begin{proof}
The proof is based on \cite{Tony11, mcontinua17}.  From the systems (\ref{eq:vartt}), (\ref{eq:var4}), $c$ as in (\ref{cabi}) and $b$ as in (\ref{bbi}), we get
\beq
\label{eq:varuHu}
\begin{split}
 c\left(\frac{\partial (\bfa{u}-\bfa{u}_h)}{\partial t},\bfa{v}\right) + b(\bfa{u}-\bfa{u}_h,\bfa{v}) = 0\,,
\end{split}
\eeq
for all $\bfa{v} \in \bfa{V}_h$.  

Given $\bfa{w} \in \bfa{V}_h$, let $\bfa{v} = \bfa{w} - \bfa{u}_h \in \bfa{V}_h$.  For the constants $C_1, C_2 > 0$ in Assumption \ref{assumeb}, from (\ref{eq:varuHu}), we obtain
\beq
\label{csi}
\begin{split}
& \frac{1}{2} \ \frac{d}{dt} \ ||\bfa{w} -\bfa{u}_h||_c^2 
+ C_2 \ \norm{\bfa{w}-\bfa{u}_h }_{a}^2  \\
& = c\bigg(\frac{\partial (\bfa{w} -\bfa{u}_h)}{\partial t},\bfa{w} - \bfa{u}_h\bigg)  +  C_2 \ \norm{\bfa{w}-\bfa{u}_h }_{a}^2  \\
& \leq c\bigg(\frac{\partial (\bfa{w} -\bfa{u}_h)}{\partial t},\bfa{w} - \bfa{u}_h\bigg) 
+b(\bfa{w}-\bfa{u}_h , \bfa{w} - \bfa{u}_h)\\
&=c\bigg(\frac{\partial (\bfa{w} -\bfa{u})}{\partial t},\bfa{w} - \bfa{u}_h\bigg) 
+b(\bfa{w}-\bfa{u} , \bfa{w} - \bfa{u}_h )\\
& \leq \left | c\bigg(\frac{\partial (\bfa{w} -\bfa{u})}{\partial t},\bfa{w} - \bfa{u}_h\bigg) \right |
+C_1\norm{\bfa{w}-\bfa{u}}_{a}\norm{\bfa{w}-\bfa{u}_h}_{a}\\
& \leq \norm{\frac{\partial (\bfa{w} -\bfa{u})}{\partial t}}_c\norm{\bfa{w}-\bfa{u}_h}_c + C_1\norm{\bfa{w}-\bfa{u}}_{a}\norm{\bfa{w}-\bfa{u}_h}_{a}\,,
\end{split}
\eeq
where the last inequality follows from the Cauchy-Schwarz inequality.

Applying Young's inequality for the right hand side of the last inequality of (\ref{csi}), we get
\beq\label{young}
\begin{split}
& \frac{1}{2} \ \frac{d}{dt} \ ||\bfa{w} -\bfa{u}_h||_c^2 
+ C_2 \ \norm{\bfa{w}-\bfa{u}_h }_{a}^2  \\
& \leq \frac{1}{2} \norm{\frac{\partial (\bfa{w} -\bfa{u})}{\partial t}}^2_c + \frac{1}{2} \norm{\bfa{w}-\bfa{u}_h}^2_c + \frac{C_1^2}{3C_2}\norm{\bfa{w}-\bfa{u}}_{a}^2 + \frac{3C_2}{4} \norm{\bfa{w}-\bfa{u}_h}_{a}^2\,.
\end{split}
\eeq
Hence,
\beq\label{1ode}
\begin{split}
& \frac{1}{2} \ \frac{d}{dt} \ ||\bfa{w} -\bfa{u}_h||_c^2 - \frac{1}{2} \norm{\bfa{w}-\bfa{u}_h}^2_c +  \frac{C_2}{4} \ \norm{\bfa{w}-\bfa{u}_h }_{a}^2  \\
& \leq \frac{1}{2} \norm{\frac{\partial (\bfa{w} -\bfa{u})}{\partial t}}^2_c +  \frac{C_1^2}{3C_2}\norm{\bfa{w}-\bfa{u}}_{a}^2 \,.
\end{split}
\eeq

Multiplying both sides of (\ref{1ode}) by multiplicative integrating factor $e^{\int (-1) \, \dt} = e^{-t}$, we obtain
\beq\label{mif}
\begin{split}
& \frac{1}{2} \left(  \left(\frac{d}{dt} \ ||\bfa{w} -\bfa{u}_h||_c^2\right)e^{-t} -  e^{-t}  \ \norm{\bfa{w}-\bfa{u}_h}^2_c \right) + e^{-t} \  \frac{C_2}{4} \   \norm{\bfa{w}-\bfa{u}_h }_{a}^2  \\
& \leq e^{-t} \ \left( \frac{1}{2} \norm{\frac{\partial (\bfa{w} -\bfa{u})}{\partial t}}^2_c +  \frac{C_1^2}{3C_2}\norm{\bfa{w}-\bfa{u}}_{a}^2\right) \,.
\end{split}
\eeq
Taking $\displaystyle \int_0^T \cdot \; \dt$ both sides of (\ref{mif}), we get
\beq\label{intT}
\begin{split}
& \frac{1}{2} \   ||\bfa{w}(\cdot,T) -\bfa{u}_h(\cdot,T)||_c^2 \ e^{-T}  + \int_0^T  e^{-t} \  \frac{C_2}{4} \ \norm{\bfa{w}-\bfa{u}_h }_{a}^2 \, \dt \\
& \leq \frac{1}{2} \   ||\bfa{w}(\cdot,0) -\bfa{u}_h(\cdot,0)||_c^2  + \int_0^T  e^{-t} \ \left( \frac{1}{2} \norm{\frac{\partial (\bfa{w} -\bfa{u})}{\partial t}}^2_c +  \frac{C_1^2}{3C_2}\norm{\bfa{w}-\bfa{u}}_{a}^2\right) \, \dt \,.
\end{split}
\eeq
Note that $e^{-T} \leq e^{-t} \leq 1, \; \forall t \in [0,T]$. Let 
\[M = \dfrac{\tu{max} \, \left \{\dfrac{1}{2}, \dfrac{C_1^2}{3C_2}\right \}} {\tu{min} \, \left \{\dfrac{e^{-T}}{2},   \dfrac{e^{-T}C_2}{4} \right \}}\,.\]
Therefore,
\beq\label{int0}
\begin{split}
&    ||\bfa{w}(\cdot,T) -\bfa{u}_h(\cdot,T)||_c^2 \  + \int_0^T  \norm{\bfa{w}-\bfa{u}_h }_{a}^2 \, \dt \\
& \leq  M  \left(   \int_0^T    \norm{\frac{\partial (\bfa{w} -\bfa{u})}{\partial t}}^2_c \, \dt +  \int_0^T \norm{\bfa{w}-\bfa{u}}_{a}^2 \, \dt + ||\bfa{w}(\cdot,0) -\bfa{u}_h(\cdot,0)||_c^2 \right) \,.
\end{split}
\eeq 
We define the initial value $\bfa{u}_h(\cdot, 0)$ such that
$c(\bfa{u}(\cdot,0),\bfa{v}) = c(\bfa{u}_h(\cdot,0),\bfa{v})$, so $||\bfa{u} (\cdot,0)-\bfa{u}_h(\cdot,0)||_c = 0$ for all $\bfa{v} \in  \bfa{V}$. 
By triangle inequality, we thus have 
\begin{equation}\label{proj}
 ||\bfa{w} (\cdot,0)-\bfa{u}_h(\cdot,0)||_c \leq ||\bfa{w} (\cdot,0)-\bfa{u}(\cdot,0)||_c\,.
\end{equation}

From (\ref{int0}) and (\ref{proj}), we obtain
\beq\label{uuh}
\bsp
 &||\bfa{u} (\cdot,T)-\bfa{u}_h(\cdot,T)||_c^2 
+\int_0^T\norm{ \bfa{u}-\bfa{u}_h}_{a}^2 \, \dt\\
& \leq
 2 \left(||\bfa{w} (\cdot,T)-\bfa{u}_h(\cdot,T)||_c^2 +  ||\bfa{w} (\cdot,T)-\bfa{u}(\cdot,T)||_c^2
+ \int_0^T \norm{\bfa{w}-\bfa{u}_h}_{a}^2 \, \dt+ \int_0^T \norm{\bfa{w}-\bfa{u}}_{a}^2 \, \dt \right)\\
& \leq
2 M \int_0^T \norm{\frac{\partial (\bfa{w} -\bfa{u})}{\partial t}}_c^2 \, \dt + 2  M \int_0^T \norm{\bfa{w}-\bfa{u}}_{a}^2 \, \dt
+  2||\bfa{w} (\cdot,T)-\bfa{u}(\cdot,T)||_c^2\\
& \; +2 \int_0^T \norm{\bfa{w}-\bfa{u}}_{a}^2 \, \dt
+  2 M||\bfa{w} (\cdot,0)-\bfa{u}(\cdot,0)||_c^2\,.
\end{split}
\eeq

To simplify the above inequality, we note that 
\begin{equation}\label{normi} 
\norm{ \int_0^T \dfrac{\partial (\bfa{w} -\bfa{u})}{\partial t} \, \dt}_c^2  \leq T \int_0^T \norm{\dfrac{\partial (\bfa{w} -\bfa{u})}{\partial t}}_c^2 \, \dt\,.
 \end{equation}

\noindent Indeed, let 
\begin{equation}\label{zdef}
\bfa{z} =\bfa{z}(\cdot)= (\bfa{w} (\cdot,T)-\bfa{u}(\cdot,T)) - (\bfa{w} (\cdot,0)-\bfa{u}(\cdot,0)) = \dd \int_0^T \dfrac{\partial (\bfa{w} -\bfa{u})}{\partial t} \, \dt\,.
\end{equation}
Then,
\begin{align*}
 \|\bfa{z}\|_c^2 &= \langle \bfa{z} , \bfa{z}\rangle_c = \bigg \langle \bfa{z} \ , \ \int_0^T \dfrac{\partial (\bfa{w} -\bfa{u})}{\partial t} \, \dt \bigg \rangle_c = \int_0^T \bigg \langle \bfa{z} \ , \ \dfrac{\partial (\bfa{w} -\bfa{u})}{\partial t} \bigg \rangle_c \, \dt\\
 & \leq \int_0^T \| \bfa{z} \|_c \, \norm{\dfrac{\partial (\bfa{w} -\bfa{u})}{\partial t}}_c \, \dt = \| \bfa{z} \|_c \int_0^T  \norm{\dfrac{\partial (\bfa{w} -\bfa{u})}{\partial t}}_c \, \dt\,.
\end{align*}
Thus, 
\begin{equation}\label{normi1}
 \|\bfa{z}\|_c \leq \int_0^T  \norm{\dfrac{\partial (\bfa{w} -\bfa{u})}{\partial t}}_c \, \dt\,.
\end{equation}

\noindent Now, by H\"{o}lder's inequality for the right hand side of (\ref{normi1}), we get
\[ \| \bfa{z} \|_c^2 \leq \left(\int_0^T \norm{\dfrac{\partial (\bfa{w} -\bfa{u})}{\partial t}}_c  \cdot 1 \, \dt \right)^2 \leq   T \left (\int_0^T\norm{\dfrac{\partial (\bfa{w} -\bfa{u})}{\partial t}}_c^2 \, \dt \right)\,, \]
 which is (\ref{normi}).  
 
 \noindent Therefore, from (\ref{zdef}), we get
\begin{align*}
 ||\bfa{w} (\cdot,T)-\bfa{u}(\cdot,T)||_c^2 
 & = \| \bfa{z} + (\bfa{w} (\cdot,0)-\bfa{u}(\cdot,0))\|_c^2\\
 &\leq 2 \| \bfa{z}\|^2_c + 2 ||\bfa{w} (\cdot,0)-\bfa{u}(\cdot,0)||_c^2\\
 &\leq 2 T \displaystyle \int_0^T \norm{\dfrac{\partial (\bfa{w} -\bfa{u})}{\partial t}}_c^2 \, \dt + 2 ||\bfa{w} (\cdot,0)-\bfa{u}(\cdot,0)||_c^2\,.
\end{align*}  

 Finally, there exists $C>0$ such that (\ref{uuh}) becomes
\beq
\bsp
 &||\bfa{u} (\cdot,T)-\bfa{u}_h(\cdot,T)||_c^2 
+\int_0^T\norm{ \bfa{u}-\bfa{u}_h}_{a}^2 \, \dt\\
& \leq
C \left(\int_0^T \norm{\frac{\partial (\bfa{w} -\bfa{u})}{\partial t}}_c^2 \, \dt + \int_0^T \norm{\bfa{w}-\bfa{u}}_{a}^2 \, \dt
+  ||\bfa{w} (\cdot,0)-\bfa{u}(\cdot,0)||_c^2\right)\,,
\end{split}
\eeq
and (\ref{lemmain1}) follows.
\end{proof}
Let us define additional bilinear forms before proceeding to the next section.
For $\bfa{u} = (u_1,u_2) \in \bfa{V}$, using notation from (\ref{syasy}), the problem (\ref{eq:vartt}) can be written as
\beq
\label{eq:varus}
\begin{split}
& \int_{\Omega}{\mathcal C}_{11} \frac{\partial u_{1}}{\partial t}v_1 \, \dx
+\int_{\Omega} \kappa_1  \nabla u_{1} \cdot \nabla v_1 \, \dx \\
& + \int_\Omega \bfa{b}_s  \cdot \nabla(u_{1}-u_{2}) v_1\, \dx
+ \int_\Omega \bfa{b}_a  \cdot \nabla(u_{1}-u_{2}) v_1\, \dx\\
& +\int_{\Omega} Q_s (u_{1}-u_{2})v_1  \, \dx
 +\int_{\Omega} Q_a (u_{1}-u_{2})v_1  \, \dx 
 = \int_{\Omega} f_1 v_1 \, \dx\,,\\
 & \int_{\Omega}{\mathcal C}_{22} \frac{\partial u_{2}}{\partial t}v_2 \, \dx
+\int_{\Omega} \kappa_2 \nabla u_{2}\cdot \nabla v_2 \, \dx\\
 & + \int_{\Omega} \bfa{b}_s  \cdot \nabla(u_{2}-u_{1})v_2  \, \dx
 - \int_{\Omega} \bfa{b}_a  \cdot \nabla(u_{2}-u_{1})v_2  \, \dx\\
 & + \int_{\Omega} Q_s (u_{2}-u_{1}) v_2  \, \dx 
 -\int_{\Omega} Q_a (u_{2}-u_{1}) v_2  \, \dx 
 =\int_{\Omega} f_2 v_2 \, \dx\,.
\end{split}
\eeq
Also, we define the following bilinear forms in $\bfa{V}\times \bfa{V}$:
\beq\label{nforms}
\bsp
&\beta((u_1,u_2),(v_1,v_2))= \int_\Omega \bfa{b}_1  \cdot \nabla(u_{1}-u_{2}) v_1\, \dx
+\int_\Omega \bfa{b}_2  \cdot \nabla(u_{2}-u_{1}) v_2\, \dx\,,\\
&q((u_1,u_2),(v_1,v_2)) = \int_{\Omega} Q_1 (u_{1}-u_{2})v_1  \, \dx
+\int_{\Omega} Q_2 (u_{2}-u_{1})v_2  \, \dx\,,\\
&q_s((u_1,u_2),(v_1,v_2)) = \int_{\Omega} Q_s (u_{1}-u_{2})v_1  \, \dx
+\int_{\Omega} Q_s (u_{2}-u_{1})v_2  \, \dx\,,\\
&q_a((u_1,u_2),(v_1,v_2)) = \int_{\Omega} Q_a (u_{1}-u_{2})v_1  \, \dx
-\int_{\Omega} Q_a (u_{2}-u_{1})v_2  \, \dx\,,\\
&a_{Q_s}((u_1,u_2),(v_1,v_2)) = a((u_1,u_2),(v_1,v_2))+q_s((u_1,u_2),(v_1,v_2))\,,\\
&b((u_1,u_2),(v_1, v_2)) = a((u_1,u_2),(v_1,v_2))+\beta((u_1,u_2),(v_1,v_2))+q((u_1,u_2),(v_1,v_2)) \,.\\
\end{split}
\eeq
Here,
\beq\label{nbis}
\bsp
&a_i^{(j)}(u_i,v_i) = \int_{\omega_j} \kappa_i  \nabla u_{i} \cdot \nabla v_i \, \dx\,,\\
&a^{(j)}((u_1,u_2),(v_1,v_2)) = a_1^{(j)}(u_1,v_1) +a_2^{(j)}(u_2,v_2) \,,\\
&a^{(j)}_{Q_s}((u_1,u_2),(v_1,v_2)) = a^{(j)}((u_1,u_2),(v_1,v_2)) + q^{(j)}_s((u_1,u_2),(v_1,v_2)),
\end{split}
\eeq
where $u_1, u_2, v_1, v_2 \in H^1_0(\omega_j) = V(\omega_j)$.
Note that $q_s(\bfa{u},\bfa{v}) = q_s(\bfa{v},\bfa{u})$. We define the norm $||\bfa{u}||_{a_{Q_s}} = a_{Q_s}(\bfa{u},\bfa{u})$.
\section{Overview of the GMsFEM}\label{view}
We refer the readers to \cite{G1} for the details of the GMsFEM, and \cite{gne, mcontinua17} for a brief overview of the GMsFEM.  Broadly speaking, solving Eq.\ (\ref{eq:main1}) on a fine grid using the standard FEM method is very expensive (due to heterogeneous coefficients). If we use coarse grid with the FEM, the solution is not accurate because of the loss of some important local information. Thus, we utilize the GMsFEM, where local problems are solved in each coarse neighborhood, to systematically construct multiscale basis functions containing local heterogenity information. More specifically, by first solving local snapshot and local eigenvalue problems, we then deduce a so-called multiscale space as global offline space $\bfa{V}_{\tu{ms}}$ (consisting of multiscale basis functions).  Hence, for all $\bfa{v}=(v_1,v_2) \in \bfa{V}_{\tu{ms}} $, the GMsFEM solution $\bfa{u}_{\tu{ms}}=(u_{\tu{ms},1}, u_{\tu{ms},2}) \, (\in \bfa{V}_{\tu{ms}} )$ is defined via the following system:
\beq
\label{eq:varuH}
\begin{split}
& \int_{\Omega}{\mathcal C}_{11} \frac{\partial u_{\tu{ms},1}}{\partial t}v_1 \, \dx
+\int_{\Omega} \kappa_1(\bfa{x})  \nabla u_{\tu{ms},1} \cdot \nabla v_1 \, \dx
+ \int_\Omega \bfa{b}_1(\bfa{x})  \cdot \nabla(u_{\tu{ms},1}-u_{\tu{ms},2}) v_1 \, \dx\\
 & \; +\int_{\Omega} Q_1(\bfa{x}) (u_{\tu{ms},1}-u_{\tu{ms},2})v_1  \, \dx 
 = \int_{\Omega} f_1 v_1 \, \dx\,,\\
 & \int_{\Omega}{\mathcal C}_{22} \frac{\partial u_{\tu{ms},2}}{\partial t} v_2 \, \dx
+\int_{\Omega} \kappa_2(\bfa{x}) \nabla u_{\tu{ms},2}\cdot \nabla v_2 \, \dx
 + \int_{\Omega} \bfa{b}_2(\bfa{x})  \cdot \nabla(u_{\tu{ms},2}-u_{\tu{ms},1})v_2  \, \dx\\
 & \; + \int_{\Omega} Q_2(\bfa{x}) (u_{\tu{ms},2}-u_{\tu{ms},1}) v_2  \, \dx 
 =\int_{\Omega} f_2 v_2 \, \dx \,.
\end{split}
\eeq
\subsection{Coarse and fine grids}
First, let $\mathcal{T}^H$ be a coarse grid, with grid size $H$.  In $\mathcal{T}^H$, each coarse block can be denoted by $K_i$.  A refinement of $\mathcal{T}^H$ is called a fine grid $\mathcal{T}_h$, with grid size $h$ ($\ll H$).  We denote by $N$ the total number of coarse blocks, and $N_v$ the total number of interior vertices of $\mathcal{T}^H$.  Let $\{\bfa{x}_i\}_{i=1}^N$ be the set of all vertices in $\mathcal{T}^H$.  The $j$th coarse neighborhood is defined by
\beq
\omega_j = \bigcup \{K_i \in \mathcal{T}^H : \bfa{x}_j \in \overline{K_i} \}.
\eeq  

Next, we will present the definitions of the uncoupled multiscale basis functions (uncoupled GMsFEM) and the coupled multiscale basis functions (coupled GMsFEM).  For each case, based on the above general procedure, we first generate a local snapshot space for each coarse neighborhood $\omega_j$, then solve an appropriate local spectral problem defined on the snapshot space, to establish a multiscale (offline) space.  There are several choices of snapshot spaces (see \cite{G1, gne}, for instance).  In this paper, for each case, its snapshot space is a set of \textbf{harmonic basis functions} (to be specified in the next subsections), which are solutions for the corresponding harmonic extension problem.  Note that the snapshot functions and the basis functions are time-independent.

\subsection{Uncoupled GMsFEM}\label{uncg}
As in \cite{gne}, let $V^i_h(\omega_j)=V_h(\omega_j)$ be a fine-scale FEM space, which is the restriction in $\omega_j$ the conforming space $V^i_h =V_h$ (introduced in Section \ref{fsfed}), for the $i$th continuum ($i=1,2$).  Let $J_h(\omega_j)$ be the set of all nodes of the fine grid $\mathcal{T}_h$ belonging to $\partial \omega_j$.  We denote by $J_j$ the cardinality of $J_h(\omega_j)$.  

For the case of uncoupled GMsFEM, multiscale basis functions will be established for each $i$th continuum separately, by taking into account only the permeability $\kappa_i$ and neglecting the transfer functions.  

More specifically, on each coarse neighborhood $\omega_j$, for each $i$th continuum, we first find the $k$th snapshot function $\phi_{k,i}^{(j),\tu{snap}} \in V_h(\omega_j)$ such that
\beq\label{snapu}
\bsp
-\div ( \kappa_i \nabla \phi_{k,i}^{(j),\tu{snap}} ) &= 0 \ \ \ \text{in} \ \omega_j,\\
\phi_{k,i}^{(j),\tu{snap}} &= \delta_{k,i} \ \ \ \text{on} \ \partial \omega_j\,,
\end{split}
\eeq
where $\delta_{k,i}$ is a discrete delta function such that 
\[\delta_{k,i}(\bfa{x}^j_l) =
 \begin{cases}
  1 \quad l = k\,,\\
  0 \quad l \ne k\,,
 \end{cases}
\]
for all $\bfa{x}^j_l$ in $J_h(\omega_j)\,,$ $1 \leq k \leq J_j$.  The solutions of this problem (\ref{snapu}) are called harmonic basis functions.
Then, the local snapshot space on $\omega_j$ for the $i$th continuum is defined as
\beq
V^i_{\tu{snap}}(\omega_j) = \text{span}\{ \phi_{k,i}^{(j),\tu{snap}} \, \bigr | \,  1 \leq k \leq J_j \}\,,
\eeq
where $J_j$ is the cardinality of $J_h(\omega_j)$ as above.

To construct local multiscale basis functions on $\omega_j$ corresponding to the $i$th continuum ($i=1,2$), we now solve local spectral problems:  Find the eigenfunctions $\psi_{k,i}^{(j)} \in V^i_{\tu{snap}}(\omega_j)$ and eigenvalues $\lambda_{k,i}^{(j)}\in \mathbb{R}$ such that
\beq\label{eeunc}
a_i^{(j)}(\psi_{k,i}^{(j)},v_i) = \lambda_{k,i}^{(j)} s_i^{(j)}(\psi_{k,i}^{(j)},v_i)\,,
\eeq
for all $v_i$ in $V^i_{\tu{snap}}(\omega_j) $, where $s_i^{(j)}$ is defined as follows (\cite{gne,mcontinua17}):
\beq\label{schi}
s_i^{(j)}(u_i,v_i) = \int_{\omega_j} \kappa_i\left( \sum_{j=1}^{N_v} |\nabla \chi_{j,i}|^2\right) u_i v_i \, \dx\,,
\eeq
where each $\chi_{j,i}$ is a standard multiscale finite element basis function for the coarse node $\bfa{x}_j$ (that is, with linear boundary
conditions for cell problems) in the $i$th continuum, and $\{\chi_{j,i}\}_{j=1}^{N_v}$ is a set of partition of unity functions (for coarse grid) supported in the intersection of $\omega_j$ and the $i$th continuum.  More specifically, based on \cite{pou},
\beq\label{chiu}
\bsp
-\div (\kappa_i \nabla \chi_{j,i}) &= 0 \ \ \textrm{in } K \in \omega_j\,,\\
\chi_{j,i} & = \chi_{j,i}^0\ \ \textrm{on} \  \ \partial K\,, \quad \forall K \in \omega_j\,,
\end{split}
\eeq
where each $\chi_{j,i}^0$ is a standard linear (and continuous) partition of unity function, and note that $\chi_{j,i}^0 = 0  \textrm{ on } \partial \omega_j\,.$



%
After sorting the eigenvalues $\lambda_{k,i}^{(j)} \, (\tu{for } k=1,2, \cdots)$ from (\ref{eeunc}) in ascending order, we choose the first corresponding $L_j$ eigenfunctions from (\ref{eeunc}), and still denote them by $\psi_{1,i}^{(j)}, \cdots, \psi_{L_j,i}^{(j)}$.  At the last step, the $k$th multiscale basis function for the $i$th continuum on the coarse neighborhood $\omega_j$ is defined by
\begin{equation}\label{mbsu}
 \psi_{k,i}^{(j),\tu{ms}} = \chi_{j,i} \psi_{k,i}^{(j)}\,,
\end{equation}
where $1 \leq k\leq L_j$, and $\{\chi_{j,i}\}_{j=1}^{N_v}$ is from (\ref{chiu}).

We define the local auxiliary offline multiscale space $V_{\tu{ms}}^{i}(\omega_j) $ for the coarse neighborhood $\omega_j$ corresponding to the $i$th continuum, using the first
$L_j$ multiscale basis functions as follows:
\beq
V_{\tu{ms}}^{i}(\omega_j) = \text{span}\left \{\psi_{k,i}^{(j),\tu{ms}} \, \bigr | \, 1 \leq k\leq L_j \right \}\,.
\eeq
Then, the global offline space for the $i$th continuum is 
\[V_{\tu{ms}}^i = \sum\limits_{j=1}^{N_v} V_{\tu{ms}}^{i}(\omega_j)= \text{span}\left \{\psi_{k,i}^{(j),\tu{ms}} \, \bigr | \, 1 \leq j \leq N_v \,, 1\leq k\leq L_j \right \}\,.\]
The multiscale space $\bfa{V}_{\tu{ms}}$ can be taken as the global offline space:  $\bfa{V}_{\tu{ms}} = V_{\tu{ms}}^{1}\times V_{\tu{ms}}^{2}\,.$
\subsection{Coupled GMsFEM}\label{cg}
In the coupled GMsFEM, the multiscale basis functions will be created by first solving a coupled problem for snapshot space, then applying a spectral decomposition.  

Note that for the case of coupled GMsFEM, the interaction terms $Q_1$ and $Q_2$ from (\ref{eq:vartt}) will be taken into account.  For eigenvalue problem, the operator should be symmetric.  Therefore, we wish to only consider the dominant symmetric part $Q_s$ (of $Q_1$ and $Q_2$) and ignore $Q_a$ from (\ref{eq:varus}), which is equivalent to (\ref{eq:vartt}).  In order to do so, we will utilize Assumption \ref{bdd} (that is, $|\bfa{b}_s| \gg |\bfa{b}_a|$ and  $|Q_s| \gg |Q_a|$) and Lemma \ref{ceamsaq2} in Section \ref{ca}.

More specifically, we find the snapshot functions $\bfa{\phi}_{k,r}^{(j),\tu{snap}} = \left (\phi_{k,1,r}^{(j),\tu{snap}},\phi_{k,2,r}^{(j),\tu{snap}}\right)$ in $\bfa{V}_h(\omega_j)=V_h(\omega_j)\times V_h(\omega_j)$ (the spaces are from Subsections \ref{fsfed} and \ref{uncg}) such that
\beq\label{snapc}
\bsp
-\div \left( \kappa_1 \nabla \phi_{k,1,r}^{(j),\tu{snap}} \right) + Q_s \left(\phi_{k,1,r}^{(j),\tu{snap}}-\phi_{k,2,r}^{(j),\tu{snap}}\right)= 0 \ \ \ \text{in} \ \omega_j,\\
-\div \left( \kappa_2 \nabla \phi_{k,2,r}^{(j),\tu{snap}} \right) + Q_s \left(\phi_{k,2,r}^{(j),\tu{snap}}-\phi_{k,1,r}^{(j),\tu{snap}}\right)= 0 \ \ \ \text{in} \ \omega_j,\\
\bfa{\phi}_{k,r}^{(j),\tu{snap}} = \bfa{\delta}_{k,r} \ \ \ \text{on} \ \partial \omega_j\,,
\end{split}
\eeq
where each $\bfa{\delta}_{k,r}$ is defined as
\beq
\bsp
\bfa{\delta}_{k,r}(\bfa{x}_l) = \delta_k(\bfa{x}_l) \bfa{e}_r, \ \ r = 1,2\,,
\end{split}
\eeq
in which $\{\bfa{e}_r \, | \, r=1,2\}$ is a standard basis in $\mathbb{R}^2\,,$ $1 \leq k \leq J_j$.  The solutions of this problem (\ref{snapc}) are called harmonic basis functions.
Then, the local snapshot space is defined as
\beq
\bfa{V}_{\tu{snap}}(\omega_j) = \text{span} \left \{ \bfa{\phi}_{k,r}^{(j),\tu{snap}} \, \bigr | \,  1\leq k \leq J_j,\ 1 \leq r \leq 2 \right \}\,.
\eeq

Next, local eigenvalue problems are solved, to construct local multiscale basis functions.  That is, we find the eigenfunctions $\bfa{\psi}_{k}^{(j)} = \left(\psi_{k,1}^{(j)}, \psi_{k,2}^{(j)}\right) \in \bfa{V}_{\tu{snap}}(\omega_j)$ and eigenvalues $\lambda_{k}^{(j)} \in \mathbb{R}$ such that
\beq\label{eec}
a_{Q_s}^{(j)}\left(\bfa{\psi}_{k}^{(j)},\bfa{v}\right) = \lambda_{k}^{(j)} s^{(j)}\left(\bfa{\psi}_{k}^{(j)},\bfa{v}\right),
\eeq
for all $\bfa{v} \in \bfa{V}_{\tu{snap}}(\omega_j) $, where $s^{(j)}$ is defined as follows (\cite{gne,mcontinua17}):
\beq\label{pouu}
s^{(j)}(\bfa{u},\bfa{v}) = \sum_{i=1}^2 s_i^{(j)}(u_i,v_i) = \sum_{i=1}^2 \int_{\omega_j} \kappa_i\left( \sum_{j=1}^{N_v} |\nabla \chi_{j,i}|^2\right) u_i v_i \, \dx\,,
\eeq
in which $\{\chi_{j,i}\}_{j=1}^{N_v}$ is from (\ref{chiu}).

After arranging the eigenvalues $\lambda_{k}^{(j)} \, (\tu{for } k=1,2,\cdots)$ from (\ref{eec}) in ascending order, we take the first corresponding $L_j$ eigenfunctions from (\ref{eec}), and still denote them by $\bfa{\psi}_1^{(j)}, \cdots,\bfa{\psi}_{L_j}^{(j)}$.  
At the final step, we define the $k$th
multiscale basis functions for the coarse region $\omega_j$ by 
\begin{equation}\label{msc}
\bfa{\psi}_{k}^{(j),\tu{ms}} =  (\chi_{j,1} \, \psi_{k,1}^{(j)} \,, \chi_{j,2} \, \psi_{k,2}^{(j)})\,,
\end{equation}
where $1 \leq k \leq L_j$, and $\{\chi_{j,i}\}_{j=1}^{N_v}$ is from (\ref{chiu}).  

The local auxiliary offline multiscale space $\bfa{V}_{\tu{ms}}(\omega_j) $ is defined by using the first $L_j$ multiscale basis functions as follows: 
\beq
\bfa{V}_{\tu{ms}}(\omega_j)= \text{span}\left \{\bfa{\psi}_{k}^{(j),\tu{ms}} \, \bigr | \, 1\leq k\leq L_j \right \}\,.
\eeq
Then, the multiscale space $\bfa{V}_{\tu{ms}}$ can be taken as the global offline space: 
\[\bfa{V}_{\tu{ms}} =\sum\limits_{j=1}^{N_v} \bfa{V}_{\tu{ms}}(\omega_j) = \text{span}\left \{\bfa{\psi}_{k}^{(j),\tu{ms}} \, \bigr | \, 1 \leq j \leq N_v \,, 1\leq k\leq L_j \right \}\,.\]
\section{Convergence Analysis (GMsFEM)}\label{ca}

In this section, we show convergence analysis for both uncoupled and coupled GMsFEM.  First, best (a-priori) error estimate is provided, for our semi-discrete problem.  We will compare the difference between the reference weak solution $\bfa{u} \in \bfa{V}$ defined in (\ref{eq:vartt}) and the multiscale solution $\bfa{u}_{\tu{ms}} \in \bfa{V}_{\tu{ms}}$ defined in (\ref{eq:varuH}), by using the projection error of $\bfa{u}$ onto $\bfa{V}_{\tu{ms}}$ in various norms.
\begin{lemma}
\label{ceamsa}
Under Assumption \ref{bdd_coer_a}, for $\bfa{u}$ and $\bfa{u}_{\tu{ms}}$ defined in (\ref{eq:vartt}) and (\ref{eq:varuH}), respectively, where $\bfa{V}_{\tu{ms}}$ is constructed via the uncoupled GMsFEM, we have the following result:
\beq\label{err1}
\bsp
& ||\bfa{u} (\cdot,T)- \bfa{u}_{\tu{ms}}(\cdot,T)||_c^2 
+ \int_0^T ||\bfa{u}-\bfa{u}_{\tu{ms}}||_{a}^2 \, \dt\\
& \leq
 C \inf\limits_{w \in \bfa{V}_{\tu{ms}}} \bigg(\int_0^T \norm{\frac{\partial (\bfa{w} - \bfa{u})}{\partial t}}_c^2 \, \dt 
 + \int_0^T \norm{\bfa{w} -\bfa{u}}_{a}^2 \, \dt
+ ||\bfa{w} (\cdot,0)- \bfa{u}(\cdot,0)||_c^2\bigg).
\end{split}
\eeq
\end{lemma}
\begin{proof}
The proof is similar to that of Lemma \ref{ceafem1}.
\end{proof}

In the spirit of this Lemma, based on \cite{mcontinua17}, to complete the convergence proof for our proposed approach, we will find an appropriate function $\bfa{w}$ in the multiscale space $\bfa{V}_{\tu{ms}}$, then estimate the error $\bfa{w} - \bfa{u}$ (the so-called projection error of $\bfa{u}$ onto $\bfa{V}_{\tu{ms}}$) in various norms on the right hand side of (\ref{err1}).  More specifically, we will define an approximation $\bfa{u}_{\tu{snap}} \in \bfa{V}_{\tu{snap}}$ (called snapshot projection) of $\bfa{u}$ in the snapshot space (which is the set of all snapshot functions).  We can express $\bfa{w} - \bfa{u} = \bfa{w} - \bfa{u}_{\tu{snap}} + \bfa{u}_{\tu{snap}} - \bfa{u}$, where the last term $\bfa{u}_{\tu{snap}} - \bfa{u}$ corresponds to an irreducible error of our method, and can be assumed to be very small by utilizing a large enough collection of snapshot functions.  It hence suffices to only estimate $\bfa{w} - \bfa{u}_{\tu{snap}}$ by choosing a suitable function $\bfa{w} \in \bfa{V}_{\tu{ms}}$.


We will define $\bfa{w} \in \bfa{V}_{\tu{ms}}$ as the projection of $\bfa{u}_{\tu{snap}}$ onto the multiscale space $\bfa{V}_{\tu{ms}}$.  In particular, first, in the case of uncoupled GMsFEM, the snapshot projection $\bfa{u}_{\tu{snap}}$ (in $\bfa{V}_{\tu{snap}}$) of $\bfa{u}$ can be represented by the set of  $\psi_{k,i}^{(j)}(\bfa{x})$ from (\ref{eeunc}) as follows:
\beq\label{usnapu}
\bsp
\bfa{u}_{\tu{snap}}(\bfa{x},t) = (u_{\tu{snap},1}, u_{\tu{snap},2}) \,, \quad u_{\tu{snap},i} = \sum\limits_{j=1}^{N_v} \sum\limits_{k} d_{k,i}^{(j)}(t) \, \chi_{j,i}(\bfa{x}) \, \psi_{k,i}^{(j)}(\bfa{x})\,.
 \end{split}
 \eeq
 We define the local component of $u^{(j)}_{\tu{snap},i}$ by
 \beq
 u^{(j)}_{\tu{snap},i}(\bfa{x},t)=\sum\limits_{k} d_{k,i}^{(j)}(t) \,  \psi_{k,i}^{(j)}(\bfa{x})\,, \quad \tu{ with } u^{(j)}_{\tu{snap},i} |_{\partial \omega_j} = u_i|_{\partial \omega_j}\,.
 \eeq
 Then, the projection $\bfa{w}$ of $\bfa{u}_{\tu{snap}}$ in the multiscale space $\bfa{V}_{\tu{ms}}$ is defined as
 \begin{equation}\label{wu}
  \bfa{w}(\bfa{x},t)= (w_1,w_2) \,, \quad w_i 
  = \sum\limits_{j=1}^{N_v} \sum\limits_{k=1}^{L_j} d_{k,i}^{(j)}(t) \, \psi_{k,i}^{(j), \tu{ms}}(\bfa{x})
  =\sum\limits_{j=1}^{N_v} \sum\limits_{k=1}^{L_j} d_{k,i}^{(j)}(t) \, \chi_{j,i}(\bfa{x}) \, \psi_{k,i}^{(j)}(\bfa{x})\,,
 \end{equation}
where the collection of local multiscale basis functions $\left \{\psi_{k,i}^{(j), \tu{ms}}(\bfa{x}) \, \bigr | \, 1\leq k \leq L_j\right \}$ is from (\ref{mbsu}).

Second, in the case of coupled GMsFEM, the snapshot projection $\bfa{u}_{\tu{snap}}$ (in $\bfa{V}_{\tu{snap}}$) of $\bfa{u}$ can be represented by the set of  $\bfa{\psi}_{k}^{(j)}(\bfa{x}) = \left(\psi_{k,1}^{(j)}(\bfa{x}), \psi_{k,2}^{(j)}(\bfa{x})\right)$ from (\ref{eec}) as follows:
\beq\label{usnapc}
\bsp
\bfa{u}_{\tu{snap}}(\bfa{x},t) = (u_{\tu{snap},1}, u_{\tu{snap},2}) \,, \quad u_{\tu{snap},i} = \sum\limits_{j=1}^{N_v} \sum\limits_{k} d_{k,i}^{(j)}(t) \, \chi_{j,i}(\bfa{x}) \, \psi_{k,i}^{(j)}(\bfa{x})\,.
 \end{split}
 \eeq
 We define the local component of $u^{(j)}_{\tu{snap},i}$ by
 \beq
 u^{(j)}_{\tu{snap},i}(\bfa{x},t)=\sum\limits_{k} d_{k,i}^{(j)}(t) \,  \psi_{k,i}^{(j)}(\bfa{x})\,, \quad \tu{ with } u^{(j)}_{\tu{snap},i} |_{\partial \omega_j} = u_i|_{\partial \omega_j}\,.
 \eeq
 Then, the projection $\bfa{w}$ of $\bfa{u}_{\tu{snap}}$ in the multiscale space $\bfa{V}_{\tu{ms}}$ is defined as
 \begin{equation}\label{wc}
  \bfa{w}(\bfa{x},t)= (w_1,w_2) \,, \quad w_i 
  = \sum\limits_{j=1}^{N_v} \sum\limits_{k=1}^{L_j} d_{k,i}^{(j)}(t) \, \psi_{k,i}^{(j), \tu{ms}}(\bfa{x})
  =\sum\limits_{j=1}^{N_v} \sum\limits_{k=1}^{L_j} d_{k,i}^{(j)}(t) \, \chi_{j,i}(\bfa{x}) \, \psi_{k,i}^{(j)}(\bfa{x})\,,
 \end{equation}
where the collection of local multiscale basis functions $\left \{\bfa{\psi}_{k}^{(j), \tu{ms}}(\bfa{x}) \, \bigr | \, 1\leq k \leq L_j\right \}$ is from (\ref{msc}).

Now, we present the main results of this section.


\subsection{Uncoupled GMsFEM}
Convergence analysis is presented for the uncoupled GMsFEM.  We will compare the difference between the reference weak solution $\bfa{u}$ defined in (\ref{eq:vartt}) and the multiscale solution $\bfa{u}_{\tu{ms}}$ defined in (\ref{eq:varuH}) from the uncoupled GMsFEM.
 
\begin{lemma}
\label{uncoupledchi}
For the uncoupled GMsFEM, if $\bfa{u}$ in (\ref{eq:vartt}) satisfies
\beq
\int_{\omega_j} \kappa_1 \nabla u_1 \cdot \nabla v_1 \, \dx
+\int_{\omega_j} \kappa_2 \nabla u_2 \cdot \nabla v_2 \, \dx
= \int_{\omega_j} f_1 v_1\, \dx + \int_{\Omega_j} f_2 v_2\, \dx\,,
\eeq
for all $\bfa{v} \in \bfa{V}(\omega_j)$, then we have
\beq\label{un1i}
\bsp
& \int_{\omega_j} \kappa_1 \chi_{j,1}^2 |\nabla u_1|^2 \, \dx
+\int_{\omega_j} \kappa_2 \chi_{j,2}^2 |\nabla u_2|^2 \, \dx\\
& \leq 
C \sum_{i=1}^2 \bigg( \int_{\omega_j} \frac{\chi_{j,i}^4}{\kappa_i |\nabla \chi_{j,i}|^2} f_i^2\, \dx 
+  \int_{\omega_j} \kappa_i |\nabla \chi_{j,i}|^2  u_i^2 \, \dx
\bigg)\,.
\end{split}
\eeq
\end{lemma}
\begin{proof}
We base on \cite{mcontinua17} for the proof.  Take $v_i = (\chi^2_{j,i})u_i \ (\tu{for } i = 1,2)$, we obtain 
\begin{align*}
\sum_{i=1}^2 \int_{\omega_j} \kappa_i (\nabla u_i) \cdot \nabla(\chi_{j,i}^2 u_i) \, \dx = \sum_{i=1}^2 \int_{\omega_j} f_i(\chi_{j,i}^2)u_i \, \dx\,. 
\end{align*}
This leads to
\begin{align*}
\sum_{i=1}^2 \int_{\omega_j} \kappa_i \chi_{j,i}^2 |\nabla u_i|^2 \, \dx & = \sum_{i=1}^2 \int_{\omega_j} f_i \frac{\chi_{j,i}^2}{(\nabla \chi_{j,i}) \sqrt{\kappa_i}} \, \sqrt{\kappa_i} u_i \nabla \chi_{j,i} \, \dx \\
& \tu{ } - 2 \sum_{i=1}^2 \int_{\omega_j} \kappa_i (\chi_{j,i}) (\nabla u_i) \cdot (\nabla \chi_{j,i}) u_i \, \dx\\
& \leq \frac{\e}{2} \sum_{i=1}^2 \int_{\omega_j} f_i^2 \, \frac{\chi_{j,i}^4}{|\nabla \chi_{j,i}|^2 \, \kappa_i} \, \dx + \frac{1}{2\e} \sum_{i=1}^2 \int_{\omega_j} \kappa_i (u_i \nabla \chi_{j,i})^2 \, \dx\\
& + \e \sum_{i=1}^2 \int_{\omega_j} \kappa_i \chi_{j,i}^2 |\nabla u_i|^2 \, \dx + \frac{1}{\e} \sum_{i=1}^2 \int_{\omega_j} \kappa_i (u_i \nabla \chi_{j,i})^2 \, \dx\,,
\end{align*}
where the last inequality follows from Young's inequality.  Let $\e = 1/2$, and move the third term on the right hand side to the left hand side of the above inequality.  Then, for some constant $C>0$, the desired inequality (\ref{un1i}) holds.
\end{proof}
We finally have the following error estimate.
\begin{theorem}
\label{mainu2}
Let $\bfa{u}$ be the solution of (\ref{eq:vartt}), $\bfa{u}_{\tu{snap}}$ and $\bfa{w}$ be defined in (\ref{usnapu}) and (\ref{wu}), respectively.  Then, we obtain the following result: 
\beq\label{convmsa}
\bsp
&  \int_0^T  \left \| \frac{\partial(\bfa{w} -\bfa{u}_{\tu{snap}})}{\partial t}\right \|_c^2 \, \dt 
+  \int_0^T  || \bfa{w}-\bfa{u}_{\tu{snap}}||_a^2 \, \dt + ||\bfa{w}(\cdot,0) - \bfa{u}_{\tu{snap}} (\cdot,0)||_c^2 \\
& \leq \frac{C}{\Lambda_1}\left( \int_0^T  \left \| \frac{\partial \bfa{u}}{\partial t}\right \|_a^2 \, \dt+
\int_0^T ||\bfa{u}||^2_a  \, \dt +||\bfa{u}(\cdot,0)||^2_a \right),
\end{split}
\eeq
where $\Lambda_1=\min\limits_{j,i} \{\lambda_{L_j+1,i}^{(j)}\}\,.$ 
\end{theorem}
\begin{proof}
We base on \cite{mcontinua17, Tony11} for the proof of this Theorem. That is, our proof follows from Lemmas \ref{partuc}, \ref{auc} and \ref{0uc} at the end of this section.

\end{proof}

\subsection{Coupled GMsFEM}
Convergence analysis is provided for the coupled GMsFEM.   We will compare the difference between the reference weak solution $\bfa{u}$ defined in (\ref{eq:vartt}) and the multiscale solution $\bfa{u}_{\tu{ms}}$ defined in (\ref{eq:varuH}) from the coupled GMsFEM. 

We will utilize the notation from (\ref{nforms}) and (\ref{nbis}).  Assume that there is some positive constant $\overline{Q_s}$ such that $|Q_s| \leq \overline{Q_s}$.  Then, it is easy to show that
\beq
\label{eq:main8}
\bigg( 1 - \frac{2\overline{Q_s} C^2_p}{\underline{\kappa}} \bigg) a (\bfa{u},\bfa{u}) \leq a_{Q_s}(\bfa{u},\bfa{u}) 
\leq \bigg( 1 + \frac{2\overline{Q_s} C^2_p}{\underline{\kappa}} \bigg) a(\bfa{u},\bfa{u})\,,
\eeq
where $C_p(\Omega)$ is from (\ref{poincare}).  We now have the following lemma.
\begin{lemma}
\label{norm_equiv}
Assume $\bigg( 1 - \dfrac{2\overline{Q_s} C^2_p}{\underline{\kappa}} \bigg) >0$. Then, there exist constants $m_1, m_2 > 0$ such that
\beq
m_1 \, a (\bfa{u},\bfa{u}) \leq a_{Q_s}(\bfa{u},\bfa{u}) 
\leq m_2 \, a(\bfa{u},\bfa{u})\,.
\eeq
\end{lemma}
Throughout this section, we always assume that $\bigg( 1 - \dfrac{2\overline{Q_s} C^2_p}{\underline{\kappa}} \bigg) >0$ holds. Recall that 
$a_{Q_s}(\bfa{u},\bfa{v}) = a_{Q_s}(\bfa{v},\bfa{u})$.

\begin{lemma}
\label{bddgarding2}
Let $K$, $\alpha$ and $C_b$ be defined as in Lemma \ref{bddgarding} and its proof.
\beq
\begin{split}
&b((u_1,u_2),(v_1,v_2)) \leq C_{b} \bigg( 1 - \dfrac{2\overline{Q_s} C^2_p}{\underline{\kappa}} \bigg)^{-1}||\bfa{u}||_{a_{Q_s}} \ ||\bfa{v}||_{a_{Q_s}}, \\
&b((u_1,u_2),(u_1,u_2)) + K ||\bfa{u}||_{L^2(\Omega)}^2 \geq \alpha \bigg( 1 + \dfrac{2\overline{Q_s} C^2_p}{\underline{\kappa}} \bigg)^{-1}||\bfa{u}||_{a_{Q_s}}^2\,,
\end{split}
\eeq
for all $(u_1,u_2), \, (v_1,v_2) \in \bfa{V}$.
\end{lemma}
\begin{proof}
The result follows from Lemma \ref{bddgarding} and (\ref{eq:main8}).
\end{proof}

The following assumption is for later theorem.

\begin{assumption}\label{assQs}
 We assume that $\alpha \bigg( 1 + \dfrac{2\overline{Q_s} C^2_p}{\underline{\kappa}} \bigg)^{-1} > \dfrac{K \, C_p}{\sqrt{\underline{\kappa}}}$ where $K$, $\alpha$ and $C_p$ are from the proof of Lemma \ref{bddgarding}. 
\end{assumption}

\begin{theorem}
\label{unique_alphaQ}
Under Assumptions \ref{bdd} and \ref{assQs}, we have a unique solution of the problem (\ref{eq:var}) with respect to $||\cdot||_{a_{Q_s}}$.
\end{theorem}
\begin{proof}
The result follows from Lemma \ref{bddgarding2}, the Poincar\'{e} inequality and the Lax-Milgram Theorem.
\end{proof}
Under Assumptions \ref{bdd} and \ref{assQs}, the following assumptions are satisfied.
\begin{assumption}
\label{bdd_coer_aq}
There exists constants $D_1$, $D_2$ $>0$ such that
\beq
\begin{split}
&b((u_1,u_2),(v_1,v_2)) \leq D_1 ||\bfa{u}||_{a_{Q_s}} \ ||\bfa{v}||_{a_{Q_s}}, \\
&b((v_1,v_2),(v_1,v_2))  \geq  D_2 ||\bfa{v}||_{a_{Q_s}}^2\,,
\end{split}
\eeq
for all $\bfa{u} = (u_1,u_2), \, \bfa{v}=(v_1,v_2) \in \bfa{V}$.
\end{assumption}
\begin{lemma}
\label{ceamsaq}
Under Assumption \ref{bdd_coer_aq},
in the coupled GMsFEM, for $\bfa{u}$ and $\bfa{u}_{\tu{ms}}$ respectively defined in (\ref{eq:vartt}) and (\ref{eq:varuH}), we have the following result:
\beq
\bsp
& \norm{\bfa{u}(\cdot,T)-\bfa{u}_{\tu{ms}}(\cdot,T)}_c^2+ \int_0^T \norm{\bfa{u}-\bfa{u}_{\tu{ms}}}_{a_{Q_s}}^2 \, \dt\\
 & \leq  {C }\inf\limits_{\bfa{w} \in \bfa{V}_{\tu{ms}}} \left( \int_0^T  \left\| \frac{\partial(\bfa{w}-\bfa{u})}{\partial t}\right\|_{c}^2 \, \dt 
+   \int_0^T  || \bfa{w}-\bfa{u}||_{a_{Q_s}}^2 \, \dt + ||\bfa{w}(\cdot,0) - \bfa{u}(\cdot,0) ||_{c}^2\right).
\end{split}
\eeq
\end{lemma}
\begin{proof}
The proof is similar to that of Lemma \ref{ceafem1}.
\end{proof}
We hence obtain the following convergence result, under weaker condition on the bilinear form $b$.
\begin{lemma}
\label{ceamsaq2}
Assume that there exist positive constants $\overline{Q_a}$, $D_1$ and $D_2$ such that $|Q_a| \leq \overline{Q_a}$ and
\beq\label{ceamsaq2e}
\begin{split}
D_2 ||\bfa{v}||_{a_{Q_s}}^2 \leq b((v_1,v_2),(v_1,v_2))  \leq  D_1 ||\bfa{v}||_{a_{Q_s}}^2\,,
\end{split}
\eeq
for all $\bfa{v} = (v_1,v_2) \in \bfa{V}$.  For $\bfa{u}$ and $\bfa{u}_{\tu{ms}}$ respectively defined in (\ref{eq:vartt}) and (\ref{eq:varuH}) from the coupled GMsFEM, the following result holds:
\beq
\bsp
 & ||\bfa{u} (\cdot,T)-\bfa{u}_{\tu{ms}}(\cdot,T)||_c^2 
+\int_0^T \norm{\bfa{u}-\bfa{u}_{\tu{ms}}}_{a_{Q_s}}^2 \, \dt\\
& \leq
 C \inf\limits_{w \in \bfa{V}_{\tu{ms}}} \bigg(\int_0^T \norm{\frac{\partial (\bfa{w} -\bfa{u})}{\partial t}}_c^2 \, \dt 
 + \int_0^T \norm{\bfa{w}-\bfa{u}}_{a_{Q_s}}^2 \dt \\
&+ \bar{b}  \int_0^T ||\nabla \bfa{w} - \nabla \bfa{u}||_{\bfa{L}^2(\Omega)}^2
 \, \dt\\
& +\overline{Q_a} \int_0^T || \bfa{w} - \bfa{ u} ||_{\bfa{L}^2(\Omega)}^2 \, \dt
+ ||\bfa{w} (\cdot,0)-\bfa{u}_{\tu{ms}}(\cdot,0)||_c^2 \bigg)\,,
\end{split}
\eeq
where $\bar{b}$ is from Assumption \ref{bdd}. 
\end{lemma}

Note that the constant $C$ in this Lemma can be different from the one in Lemma \ref{ceamsaq}.

\begin{proof}
Recall that for all $\bfa{v} = (v_1,v_2) \in \bfa{V}_{\tu{ms}}$, from (\ref{eq:vartt}) and (\ref{eq:varuH}), we have
\beq\label{eq:varums1}
\begin{split}
 c\left(\frac{\partial (\bfa{u}-\bfa{u}_{\tu{ms}})}{\partial t},\bfa{v}\right) + b(\bfa{u}-\bfa{u}_{\tu{ms}},\bfa{v}) = 0\,.
\end{split}
\eeq

Given $\bfa{w} \in \bfa{V}_{\tu{ms}}$, we let $\bfa{v} = \bfa{w} - \bfa{u}_{\tu{ms}} \in \bfa{V}_{\tu{ms}}$.
Using notation from (\ref{nforms}), and Young's inequality, we note that
\begin{align}\label{betain}
\begin{split}
 \beta(\bfa{w} - \bfa{u}, \bfa{w} - \bfa{u}_{\tu{ms}}) & = \beta((w_1 - u_1,w_2 - u_2), (w_1 - u_{\tu{ms},1}, w_2 - u_{\tu{ms},2})) \\
 &= \int_{\Omega} \bfa{b}_1 \cdot \nabla((w_1 - u_1) - (w_2 - u_2))(w_1 - u_{\tu{ms},1}) \, \dx \\
 & \; + \int_{\Omega} \bfa{b}_2 \cdot \nabla((w_2 - u_2) - (w_1 - u_1))(w_2 - u_{\tu{ms},2}) \, \dx\\
 & \leq \frac{1}{c_1} \bar{b} ||\nabla \bfa{w} - \nabla \bfa{u} ||_{\bfa{L}^2(\Omega)}^2 +\frac{c_1}{2} ||\bfa{w} - \bfa{ u_{\tu{ms}}} ||_{\bfa{L}^2(\Omega)}^2\,,
 \end{split}
\end{align}
for some $c_1 >0\,.$
Also,
\begin{align}\label{qain}
 \begin{split}
 q_a(\bfa{w} - \bfa{u}, \bfa{w} - \bfa{u}_{\tu{ms}}) 
  \leq \frac{1}{d_1}  \overline{Q_a} \; || \bfa{w} -  \bfa{u} ||_{\bfa{L}^2(\Omega)}^2+
   \frac{d_1}{2} ||\bfa{w} - \bfa{u}_{\tu{ms}} ||_{\bfa{L}^2(\Omega)}^2\,,
   \end{split}
\end{align}
for some $d_1 >0\,.$  Hence, for $D_2$ from (\ref{ceamsaq2e}), utilizing (\ref{eq:varums1}), we obtain
\beq
\label{eq:varuHu4}
\begin{split}
&\frac{1}{2} \, \frac{d}{dt} ||\bfa{w} - \bfa{u}_{\tu{ms}} ||_c^2 + D_2 || \bfa{w} - \bfa{u}_{\tu{ms}} ||^2_{a_{Q_s}}\\
 & = c\left( \frac{\partial(\bfa{w} - \bfa{u}_{\tu{ms}})}{\partial t}, \bfa{w} - \bfa{u}_{\tu{ms}}\right) +D_2 || \bfa{w} - \bfa{u}_{\tu{ms}} ||_{a_{Q_s}}\\
& \leq c\left( \frac{\partial(\bfa{w} - \bfa{u}_{\tu{ms}})}{\partial t}, \bfa{w} - \bfa{u}_{\tu{ms}}\right) + b(\bfa{w} - \bfa{u}_{\tu{ms}},\bfa{w} - \bfa{u}_{\tu{ms}})\\
&= c\bigg(\frac{\partial (\bfa{w} -\bfa{u}_{\tu{ms}})}{\partial t},\bfa{w} - \bfa{u}_{\tu{ms}}\bigg) 
+a(\bfa{w}-\bfa{u}_{\tu{ms}} , \bfa{w} - \bfa{u}_{\tu{ms}} )
+\beta(\bfa{w}-\bfa{u}_{\tu{ms}},\bfa{w}-\bfa{u}_{\tu{ms}})\\
&\;+ q(\bfa{w}-\bfa{u}_{\tu{ms}},\bfa{w}-\bfa{u}_{\tu{ms}})\\
&=c\bigg(\frac{\partial (\bfa{w} -\bfa{u})}{\partial t},\bfa{w} - \bfa{u}_{\tu{ms}}\bigg) 
+a_{Q_s}(\bfa{w}-\bfa{u} , \bfa{w} - \bfa{u}_{\tu{ms}} )
+\beta(\bfa{w}-\bfa{u},\bfa{w}-\bfa{u}_{\tu{ms}})\\
&\;+ q_a(\bfa{w}-\bfa{u},\bfa{w}-\bfa{u}_{\tu{ms}})\\
 & \leq \bigg | c\bigg(\frac{\partial (\bfa{w} -\bfa{u})}{\partial t},\bfa{w} - \bfa{u}_{\tu{ms}}\bigg) \bigg | +\norm{\bfa{w}-\bfa{u}}_{a_{Q_s}}\norm{\bfa{w}-\bfa{u}_{\tu{ms}}}_{a_{Q_s}}
+\beta(\bfa{w}-\bfa{u},\bfa{w}-\bfa{u}_{\tu{ms}})\\
&\;+q_a(\bfa{w}-\bfa{u},\bfa{w}-\bfa{u}_{\tu{ms}})\\
& \leq \left \| \frac{\partial (\bfa{w} -\bfa{u})}{\partial t} \right \|_c \, \|\bfa{w} - \bfa{u}_{\tu{ms}}\|_c +\norm{\bfa{w}-\bfa{u}}_{a_{Q_s}}\norm{\bfa{w}-\bfa{u}_{\tu{ms}}}_{a_{Q_s}}\\
& \; +\frac{1}{c_1} \bar{b} ||\nabla \bfa{w} - \nabla \bfa{u} ||_{\bfa{L}^2(\Omega)}^2 +\frac{c_1}{2} ||\bfa{w} - \bfa{ u_{\tu{ms}}} ||_{\bfa{L}^2(\Omega)}^2 \\
 & \; +\frac{1}{d_1}  \overline{Q_a} \; || \bfa{w} -  \bfa{u} ||_{\bfa{L}^2(\Omega)}^2+
   \frac{d_1}{2} ||\bfa{w} - \bfa{u}_{\tu{ms}} ||_{\bfa{L}^2(\Omega)}^2\,,
\end{split}
\eeq
where the last inequality follows from (\ref{betain}) and (\ref{qain}).


From the Poincar\'{e} inequality (\ref{poincare}), there exists $C_p, D >0$ such that $ \|\bfa{z}\|_{\bfa{L}^2(\Omega)}^2 \leq C_p^2\|\nabla \bfa{z}\|_{\bfa{L}^2(\Omega)}^2 \leq D \|\bfa{z}\|_{a_{Q_s}}^2\,, \forall \bfa{z} \in \bfa{V} \subset \bfa{L}^2(\Omega)$.  Thus, in the last inequality of (\ref{eq:varuHu4}),
\[\frac{c_1 + d_1}{2} ||\bfa{w} - \bfa{u}_{\tu{ms}} ||_{\bfa{L}^2(\Omega)}^2 \leq \frac{D(c_1 + d_1)}{2}\|\bfa{w} - \bfa{u}_{\tu{ms}}\|_{a_{Q_s}}^2\,.\]

We define the initial value $\bfa{u}_{\tu{ms}}(\cdot, 0)$ such that
$c(\bfa{u}(\cdot,0),\bfa{v}) = c(\bfa{u}_{\tu{ms}}(\cdot,0),\bfa{v})$, so $||\bfa{u} (\cdot,0)-\bfa{u}_{\tu{ms}}(\cdot,0)||_c = 0$ for all $\bfa{v} \in  \bfa{V}$.  Then, the rest of the proof is similar to that of Lemma \ref{ceafem1}.
\end{proof}

 
\begin{lemma}
\label{coupledchi}
For the coupled GMsFEM, if $\bfa{u}$ from (\ref{eq:vartt}) satisfies
\beq
\int_{\omega_j} \kappa_1 \nabla u_1 \cdot \nabla v_1 \, \dx
+\int_{\omega_j} \kappa_2 \nabla u_2 \cdot \nabla v_2 \, \dx
+q_s(\bfa{u},\bfa{v}) = \int_{\omega_j} f_1 v_1 \, \dx + \int_{\omega_j} f_2 v_2 \, \dx\,,
\eeq
for all $\bfa{v} \in \bfa{V}(\omega_j)$, we have
\beq
\bsp
& \int_{\omega_j} \kappa_1 \chi_j^2 |\nabla u_1|^2 \, \dx
+\int_{\omega_j} \kappa_2 \chi_j^2 |\nabla u_2|^2 \, \dx
+q_s((u_1,u_2),(\chi_{j,1}^2 \, u_1, \chi_{j,2}^2 \, u_2))\\
& \leq 
C \sum_{i=1}^2 \bigg(\int_{\omega_j} \frac{\chi_j^4}{\kappa_i |\nabla \chi_j|^2} f_i^2\, \dx  
+\int_{\omega_j} \kappa_i |\nabla \chi_j|^2  u_i^2 \, \dx \bigg).
\end{split}
\eeq
\end{lemma}
\begin{proof}
The proof of this Lemma readily follows from that of Lemma \ref{uncoupledchi} and thanks to \cite{mcontinua17}.
\end{proof}

\begin{theorem}\label{mainc2}
Let $\bfa{u}$ be the solution of (\ref{eq:vartt}), $\bfa{u}_{\tu{snap}}$ and $\bfa{w}$ be defined in (\ref{usnapc}) and (\ref{wc}), respectively.  Then, we have the following estimate:
\beq\label{convmsaq}
\bsp
&  \int_0^T  \left \|\frac{\partial ( \bfa{w}-\bfa{u}_{\tu{snap}})}{\partial t}\right \|_{c}^2 \, \dt 
  \int_0^T  || \bfa{w}-\bfa{u}_{\tu{snap}}||_{a_{Q_s}}^2 \, \dt + ||\bfa{w}(\cdot,0) - \bfa{u}_{\tu{snap}}(\cdot, 0)||_{c}^2 \\
& \leq \frac{C}{\Lambda_2}\left( \int_0^T \left \|\frac{\partial \bfa{u}}{\partial t}\right \|^2_{a_{Q_s}} \, \dt +
\int_0^T ||\bfa{u}||^2_{a_{Q_s}} \, \dt +||\bfa{u}(\cdot,0)||^2_{a_{Q_s}} \right)\,,
\end{split}
\eeq
where $\Lambda_2 = \dd \min_j \{ \lambda_{L_j +1}^{(j)} \} \,.$
\end{theorem}
\begin{proof}
Following the proof in \cite{mcontinua17, Tony11}, our proof is derived from Lemmas \ref{partuc}, \ref{auc} and \ref{0uc}. 
\end{proof}

\subsection{Lemmas for the main convergence results}\label{lems}
In this part, we provide and prove some Lemmas that Theorems \ref{mainu2} and \ref{mainc2} directly follow from. 

\begin{lemma}\label{partuc}
 Let $\bfa{u}$, $\bfa{u}_{\tu{snap}}$, $\bfa{w}$, $\Lambda_1$ and $\Lambda_2$ be defined in Theorems \ref{mainu2} and \ref{mainc2}.  For the uncoupled GMsFEM, we have
 \begin{equation}\label{partu}
  \int_0^T \left \| \frac{\partial(\bfa{w} - \bfa{u}_{\tu{snap}})}{\partial t} \right \|^2_c \, \dt \leq \frac{C}{\Lambda_1} \int_0^T \left \| \frac{\partial \bfa{u}}{\partial t} \right \|^2_a \, \dt\,.
 \end{equation}
 For the coupled GMsFEM, we have
 \begin{equation}\label{partc}
  \int_0^T \left \| \frac{\partial(\bfa{w} - \bfa{u}_{\tu{snap}})}{\partial t} \right \|^2_c \, \dt \leq \frac{C}{\Lambda_2} \int_0^T \left \| \frac{\partial \bfa{u}}{\partial t} \right \|^2_{a_{Q_s}} \, \dt\,.
  \end{equation}
\end{lemma}

\begin{proof}
 Based on \cite{mcontinua17, Tony11}, we will first derive the proof for the case of uncoupled GMsFEM.  Note that
 \begin{align}\label{partu1}
  \begin{split}
  & \left \| \frac{\partial \bfa{w}}{\partial t}  - \frac{\partial \bfa{u}_{\tu{snap}}}{\partial t} \right\|^2_c \\
  & = \sum_{i=1}^{2} \int_{\Omega} \mathcal{C}_{ii} \left (\frac{ \partial w_i}{\partial t} - \frac{\partial u_{\tu{snap},i}}{\partial t} \right)^2 \, \dx\\
  & = \sum_{i=1}^2 \int_{\Omega} \mathcal{C}_{ii} \left ( \sum_{j=1}^{N_v} \sum_{k> L_j} \frac{\partial d_{k,i}^{(j)} (t)} {\partial t} \chi_{j,i}(\bfa{x}) \psi_{k,i}^{(j)}(\bfa{x}) \right)^2 \, 
  \dx\\
 & \leq C \sum_{i=1}^2 \sum_{j=1}^{N_v} \int_{\omega_j} \kappa_i\left( \sum_{j=1}^{N_v} |\nabla \chi_{j,i}|^2\right) \left( \sum_{k> L_j} \frac{\partial d_{k,i}^{(j)} (t)} {\partial t} \psi_{k,i}^{(j)}(\bfa{x}) \right)^2 \, \dx\\
 &= C \sum_{i=1}^2 \sum_{j=1}^{N_v} s_i^{(j)} \left( \sum_{k> L_j} \frac{\partial d_{k,i}^{(j)} (t)} {\partial t} \psi_{k,i}^{(j)}(\bfa{x}) , \sum_{k> L_j} \frac{\partial d_{k,i}^{(j)} (t)} {\partial t} \psi_{k,i}^{(j)}(\bfa{x})\right)\,.
 \end{split}
 \end{align}
 
 By the spectral problem (\ref{eeunc}) and the orthogonality of eigenfunctions $\{\psi_{k,i}^{(j)}(\bfa{x})\}_k$, we have
 \begin{align}\label{partu2}
 \begin{split}
  & s_i^{(j)}\left(\sum_{k> L_j} \frac{\partial d_{k,i}^{(j)} (t)} {\partial t} \psi_{k,i}^{(j)}(\bfa{x}), \sum_{k> L_j} \frac{\partial d_{k,i}^{(j)} (t)} {\partial t} \psi_{k,i}^{(j)}(\bfa{x})\right)\\
  & \leq \frac{1}{\lambda_{L_j+1,i}^{(j)}} a_i^{(j)} \left(\sum_{k> L_j} \frac{\partial d_{k,i}^{(j)} (t)} {\partial t} \psi_{k,i}^{(j)}(\bfa{x}), \sum_{k> L_j} \frac{\partial d_{k,i}^{(j)} (t)} {\partial t} \psi_{k,i}^{(j)}(\bfa{x})\right) \\
  & \leq \frac{1}{\lambda_{L_j+1,i}^{(j)}} a_i^{(j)} \left( \sum_{k} \frac{\partial d_{k,i}^{(j)} (t)} {\partial t} \psi_{k,i}^{(j)}(\bfa{x}), \sum_{k} \frac{\partial d_{k,i}^{(j)} (t)} {\partial t} \psi_{k,i}^{(j)}(\bfa{x})\right)\\
  &= \frac{1}{\lambda_{L_j +1,i}^j} a_i^{(j)} \left( \frac{\partial u^{(j)}_{\tu{snap},i}}{\partial t},\frac{\partial u^{(j)}_{\tu{snap},i}}{\partial t} \right)\,.
  \end{split}
 \end{align}
 Therefore, (\ref{partu1}) becomes 
 \begin{equation}\label{partu3}
   \left \| \frac{\partial \bfa{w}}{\partial t}  - \frac{\partial \bfa{u}_{\tu{snap}}}{\partial t} \right\|^2_c \leq C \sum_{i=1}^{2} \sum_{j=1}^{N_v} \frac{1}{\lambda_{L_j +1,i}^j} a_i^{(j)} \left( \frac{\partial u^{(j)}_{\tu{snap},i}}{\partial t},\frac{\partial u^{(j)}_{\tu{snap},i}}{\partial t} \right)\,. 
 \end{equation}
 Since $u^{(j)}_{\tu{snap},i}$ is the projection of $u_i$ in each $\omega_j$ by the definition (\ref{usnapu}), it follows that 
 \[a_i^{(j)}(u_i,v_i) = a_i^{(j)} \left(u^{(j)}_{\tu{snap},i},v_i\right)\,, \quad \forall v_i \in V_{\tu{snap}}^{i} (\omega_j)\,.\]
 More specifically, let $v_i = u_{\tu{snap},i}^{(j)}$, we have
 \[a_i^{(j)}\left(u_{\tu{snap},i}^{(j)},u_{\tu{snap},i}^{(j)}\right) = a_i^{(j)} \left(u_i,u_{\tu{snap},i}^{(j)}\right)\,, \]
 \[\norm{u^{(j)}_{\tu{snap},i}}_{a^{(j)}_i}^2 \leq \norm{u_i}_{a^{(j)}_i} \, \norm{u_{\tu{snap},i}^{(j)}}_{a^{(j)}_i}\,.\]
 Hence,
 \begin{equation}\label{proju}
  a_i^{(j)}\left(u_{\tu{snap},i}^{(j)},u_{\tu{snap},i}^{(j)}\right) \leq a_i^{(j)} (u_i,u_i)\,.
 \end{equation}
 Similarly,
 \[a_i^{(j)}\left(\frac{\partial u_{\tu{snap},i}^{(j)}}{\partial t},\frac{\partial u_{\tu{snap},i}^{(j)}} {\partial t}\right) \leq a_i^{(j)} \left(\frac{\partial u_i}{\partial t},\frac{\partial u_i}{\partial t}\right)\,.\]
 Thus, from (\ref{partu3}), we get
 \begin{align}\label{partu4}
  \begin{split}
  \left \| \frac{\partial (\bfa{w}  -  \bfa{u}_{\tu{snap}})}{\partial t} \right\|^2_c 
  & \leq C \sum_{i=1}^{2} \sum_{j=1}^{N_v} \frac{1}{\lambda_{L_j +1,i}^{(j)}} a_i^{(j)} \left(\frac{\partial u_i}{\partial t},\frac{\partial u_i}{\partial t}\right)\\
  &\leq \frac{C}{\dd \min_{j,i}\{\lambda_{L_j +1,i}^{(j)}\}} a\left( \frac{\partial \bfa{u}}{\partial t},\frac{\partial \bfa{u}}{\partial t}\right)\\
  & = \frac{C}{\dd \min_{j,i}\{\lambda_{L_j +1,i}^{(j)}\}}  \left \|\frac{\partial \bfa{u}}{\partial t} \right \|^2_a\,.
  \end{split}
 \end{align}
 
 For the case of coupled GMsFEM, recall that $s^{(j)}(\cdot,\cdot) = \sum_{i=1}^{2} s_i^{(j)}(\cdot,\cdot)$.  Applying the same arguments, we get
 \[\left \| \frac{\partial (\bfa{w}  -  \bfa{u}_{\tu{snap}})}{\partial t} \right\|^2_c  \leq C \sum_{i=1}^2 \sum_{j=1}^{N_v} s_i^{(j)} \left( \sum_{k> L_j} \frac{\partial d_{k,i}^{(j)} (t)} {\partial t} \psi_{k,i}^{(j)}(\bfa{x}) , \sum_{k> L_j} \frac{\partial d_{k,i}^{(j)} (t)} {\partial t} \psi_{k,i}^{(j)}(\bfa{x})\right)\,.\]
\end{proof}

\bigskip

\begin{lemma}\label{auc}
 Let $\bfa{u}$, $\bfa{u}_{\tu{snap}}$, $\bfa{w}$, $\Lambda_1$ and $\Lambda_2$ be defined in Theorems \ref{mainu2} and \ref{mainc2}.  For the uncoupled GMsFEM, we have
 \begin{equation}\label{au}
  \int_0^T \left \| \bfa{w} - \bfa{u}_{\tu{snap}}\right \|^2_a\, \dt \leq \frac{C}{\Lambda_1}\int_0^T \left \|  \bfa{u} \right \|^2_{a} \, \dt\,.
 \end{equation}
 For the coupled GMsFEM, we have
 \begin{equation}\label{ac}
  \int_0^T \left \| \bfa{w} - \bfa{u}_{\tu{snap}}\right \|^2_{a_{Q_s}}\, \dt \leq \frac{C}{\Lambda_2} \int_0^T \left \|  \bfa{u} \right \|^2_{a_{Q_s}} \, \dt\,.
  \end{equation}
\end{lemma}

\begin{proof}
This Lemma's proof is based on \cite{mcontinua17,Tony11}.

For the case of uncoupled GMsFEM, we define 
\[e_i^{(j)} = \sum_{k> L_j}  d_{k,i}^{(j)}(t) \, \psi_{k,i}^{(j)}(\bfa{x})\,.\]
 By (\ref{wu}) and (\ref{usnapu}), we have
 \begin{align}\label{au1}
  \begin{split}
  & \norm{\bfa{w} - \bfa{u}_{\tu{snap}}}^2_a \\
  & = \sum_{i=1}^{2} \int_{\Omega} \kappa_i \left | \nabla(w_i - u_{\tu{snap},i}) \right |^2 \, \dx \\
  & = \sum_{i=1}^2 \int_{\Omega} \kappa_i \left |  \sum_{j=1}^{N_v} \sum_{k> L_j} \nabla\left( d_{k,i}^{(j)}(t) \, \chi_{j,i} \, \psi_{k,i}^{(j)}\right)\right |^2 \, \dx\\
  & \leq N_v \sum_{i=1}^2 \sum_{j=1}^{N_v} \int_{\omega_j} \kappa_i \left | \nabla \left(\chi_{j,i} e_i^{(j)}\right) \right |^2 \, \dx\\
  &\leq 2 N_v \sum_{i=1}^2 \sum_{j=1}^{N_v}\left( \int_{\omega_j} \kappa_i \left | \nabla \chi_{j,i} \right|^2 \left | e_i^{(j)} \right |^2 \, \dx + \int_{\omega_j} \kappa_i \left |  \chi_{j,i} \right|^2 \left | \nabla e_i^{(j)} \right |^2 \, \dx \right)\,.
  \end{split}
 \end{align}
Note that  
\[  \int_{\omega_j} \kappa_i \left | \nabla \chi_{j,i} \right|^2 \left | e_i^{(j)} \right |^2 \, \dx \leq    \int_{\omega_j} \kappa_i \left( \sum_{j=1}^{N_v} |\nabla \chi_{j,i}|^2\right) \left | e_i^{(j)} \right |^2 \, \dx =  s_i^{(j)} \left(e_i^{(j)},e_i^{(j)}\right)\,.  \]
From this, Lemma \ref{uncoupledchi} and (\ref{snapu}), there exists some positive constant $D_3$ such that
\begin{align*}
 \int_{\omega_j} \kappa_i \left |  \chi_{j,i} \right|^2 \left | \nabla e_i^{(j)} \right |^2 \, \dx \leq D_3 \int_{\omega_j} \kappa_i \left | \nabla \chi_{j,i} \right|^2 \left | e_i^{(j)} \right |^2 
 \leq D_3 s_i^{(j)} \left(e_i^{(j)},e_i^{(j)}\right)\,. 
\end{align*}
Therefore,
\[\|\bfa{w} - \bfa{u}_{\tu{snap}} \|_a^2  \leq D_4  \sum_{i=1}^2 \sum_{j=1}^{N_v} s_i^{(j)} \left(e_i^{(j)},e_i^{(j)}\right)\,. \]
Finally, based on bilinearity of $a_i^{(j)}$ and $s_i^{(j)}$ as well as the orthogonality of $\{\psi_{k,i}^j \}_k$, and the definition of the eigenprojection, for the case of uncoupled GMsFEM, as in (\ref{proju}), we get
\begin{align*}
s_i^{(j)} \left(e_i^{(j)},e_i^{(j)}\right) &\leq \frac{1}{\lambda_{L_j+1,i}^{(j)}} a_i^{(j)}\left(e_i^{(j)},e_i^{(j)}\right) \leq \frac{1}{\lambda_{L_j+1,i}^{(j)}} a_i^{(j)} \left(u^{(j)}_{\tu{snap},i}, u^{(j)}_{\tu{snap},i}\right) \\
& \leq \frac{1}{\lambda_{L_j+1,i}^{(j)}} a_i^{(j)} \left(u_{i}, u_{i}\right)\,.
\end{align*}
Hence, the desired result (\ref{au}) follows.

For the case of coupled GMsFEM, similar arguments are applied for
\[\bfa{e}^{(j)} = \sum_{k> L_j}  d_{k}^{(j)}(t) \, \bfa{\psi}_{k}^{(j)}(\bfa{x})\,.\]
\end{proof}


\bigskip

\begin{lemma}\label{0uc}
 Let $\bfa{u}$, $\bfa{u}_{\tu{snap}}$, $\bfa{w}$, $\Lambda_1$ and $\Lambda_2$ be defined in Theorems \ref{mainu2} and \ref{mainc2}.  For the uncoupled GMsFEM, we have
 \begin{equation}\label{0u}
  \left \| \bfa{w}(\cdot,0) - \bfa{u}_{\tu{snap}}(\cdot,0)\right \|^2_c \leq \frac{C}{\Lambda_1} \left \|  \bfa{u}(\cdot,0) \right \|^2_{a} \,.
 \end{equation}
 For the coupled GMsFEM, we have
 \begin{equation}\label{0c}
   \left \| \bfa{w}(\cdot,0) - \bfa{u}_{\tu{snap}}(\cdot,0)\right \|^2_c\leq \frac{C}{\Lambda_2} \left \|  \bfa{u}(\cdot,0) \right \|^2_{a_{Q_s}} \,.
  \end{equation}
\end{lemma}

\begin{proof}
 For the case of uncoupled GMsFEM, as in Lemma \ref{auc}, we let
 \[e_{0,i}^{(j)} = \sum_{k> L_j} d_{k,i}^{(j)} (0) \psi_{k,i}^{(j)} (\bfa{x})\,.\]
 Then, following the proof of Lemma \ref{auc}, we get 
 \begin{align}\label{0u1}
  \begin{split}
   & \left \| \bfa{w}(\cdot,0) - \bfa{u}_{\tu{snap}}(\cdot,0)\right \|^2_c \\
   & = \sum_{i=1}^2 \int_{\Omega} \mathcal{C}_{ii} \left | u_{\tu{snap},i}(\cdot,0) -w_i(\cdot,0) \right |^2 \, \dx\\
   & = \sum_{i=1}^2 \int_{\Omega} \mathcal{C}_{ii} \left | \sum_{j=1}^{N_v} \chi_{j,i} \, e_{0,i}^{(j)}  \right |^2 \, \dx\\
   &= \left \| \sum_{j=1}^{N_v} \chi_{j,i} \, e_0^{(j)} \right \|^2_c\\
   &\leq D_6 \sum_{j=1}^{N_v} \left \| e_0^{(j)} \right \|^2_c \\
   & \leq D_6 \, D_7 \frac{1}{\Lambda_1} \sum_{i=1}^2 \sum_{j=1}^{N_v} a_i^{(j)} \left(e_{0,i}^{(j)},e_{0,i}^{(j)}\right)\\
   &  \leq D_6 \, D_7 \frac{1}{\Lambda_1} \sum_{i=1}^2 \sum_{j=1}^{N_v} a_i^{(j)} \left(u_{\tu{snap},i}^{(j)}(\cdot, 0), u_{\tu{snap},i}^{(j)}(\cdot,0)\right) \\
   & \leq D_6 \, D_7 \frac{1}{\Lambda_1} \sum_{i=1}^2 \sum_{j=1}^{N_v} a_i^{(j)} \left(u_{i}(\cdot, 0), u_{i}(\cdot,0)\right) \\
   & \leq \frac{C}{\Lambda_1} \left \|  \bfa{u}(\cdot,0) \right \|^2_{a} \,.
  \end{split}
 \end{align}
 For the case of coupled GMsFEM, similar arguments are utilized for
 \[\bfa{e}_0^{(j)} = \sum_{k> L_j} d_k^{(j)}(0) \bfa{\psi}_k^{(j)}(\bfa{x})\,.\]
\end{proof}

\section{Numerical results}\label{numer}
In this section, we present numerical results for both coupled and uncoupled GMsFEM. Let $\Omega = [0,1]^2$, and consider the following problem:
\beq
\label{eq:main_num}
\begin{split}
&\frac{\partial u_{1}}{\partial t} (\bfa{x},t)-\div (\kappa_1 (\bfa{x}) \nabla u_1(\bfa{x},t)) + \bfa{b}_1(\bfa{x})  \cdot \nabla(u_1(\bfa{x},t)-u_2(\bfa{x},t)) \\
& \; + Q_1 (u_1(\bfa{x},t)-u_2(\bfa{x},t)) = 1 \,,\\
&\frac{\partial u_{2}}{\partial t} (\bfa{x},t)-\div (\kappa_2 (\bfa{x}) \nabla u_2(\bfa{x},t)) + \bfa{b}_2(\bfa{x})  \cdot \nabla(u_2(\bfa{x},t)-u_1(\bfa{x},t)) \\
&\; + Q_2 (u_2(\bfa{x},t)-u_1(\bfa{x},t)) = 1,
\end{split}
\eeq
where we let
\beq
\bsp
&\bfa{b}_1(\bfa{x}) =10\,((1-\cos(2 \pi x_1)) \sin(2 \pi x_2 ), -\sin(2 \pi x_1)(1-\cos(2 \pi x_2))),\\
&\bfa{b}_2(\bfa{x}) = 10\,( -\sin(2 \pi x_1)(1-\cos(2 \pi x_2)), (1-\cos(2 \pi x_1)) \sin(2 \pi x_2 )).
\end{split}
\eeq
Figs.\ \ref{perm1} and \ref{perm2} indicate that the high-contrast permeability coefficients $\kappa_1$ and $\kappa_2 $ are used.
We compare the fine-scale solutions with the multiscale ones, by computing relative errors in weighted $L^2$ norm and $H^1$ semi-norm. In particular, we use 
\beq\label{errors}
\bsp
100\, ||u_{\tu{ms},i}-u_{h,i}||_{L^2_{a_i}} / ||u_{h,i}||_{L^2_{a_i}}\,,\ 100\, \|u_{\tu{ms},i}-u_{h,i}\|_{H^1_{a_i}} / \|u_{h,i}\|_{H^1_{a_i}},
\end{split}
\eeq
where $||u_i||_{L^2_{a_i}} = \displaystyle \int_\Omega \kappa_i u_i^2 \, \dx$\,, $\|u_i\|_{H^1_{a_i}} = \displaystyle \int_\Omega \kappa_i |\nabla u_i|^2 \, \dx$ (for $i=1,2$).

We denote by $DOF_{\tu{fine}}$ the number of degrees of freedom (basis functions) for fine-scale FEM.
Tables \ref{errors1}, \ref{errors2}, \ref{errors3} and \ref{errors4} represent the errors obtained from the coupled and uncoupled GMsFEM with various $Q_1$ and $Q_2$ (see Figs.\ \ref{perm_Q} and \ref{perm_c}).  From Tables \ref{errors1} and \ref{errors2}, we observe that the coupled GMsFEM has higher accuracy compared with the uncoupled GMsFEM, when $Q_1$ and $Q_2$ are large and positive.  Tables \ref{errors3} and \ref{errors4} show that both of the coupled and uncoupled GMsFEM still have good convergence with some negative $Q_1$ and $Q_2$.  Fig.\ \ref{Solution_plot} represents solutions $u_1$ obtained from the FEM and GMsFEM. 

\begin{figure}[H]
	\centering
	\begin{subfigure}{0.45\textwidth}
  \includegraphics[width=\textwidth]{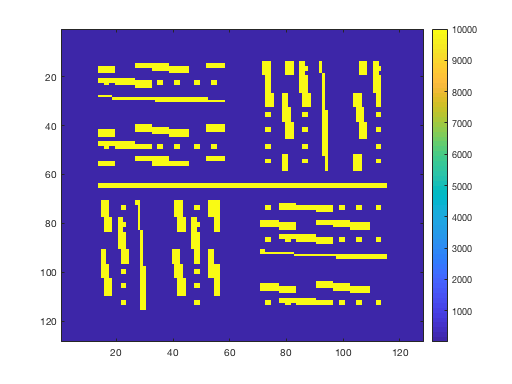}
  \caption{$\kappa_1(\bfa{x})$. The value in each channel is $10^4$.}
  \label{perm1}
\end{subfigure}
\hfill
   \begin{subfigure}{0.45\textwidth}
  \includegraphics[width=\textwidth]{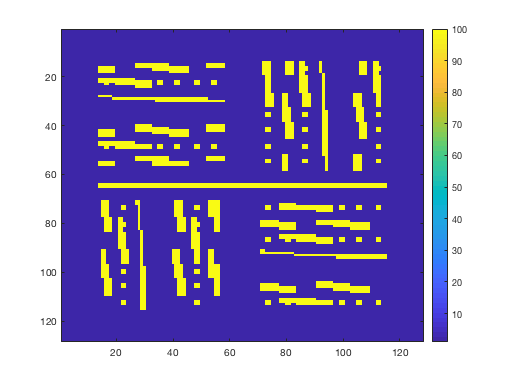}
  \caption{$\kappa_2(\bfa{x})$. The value in each channel is 100.}
  \label{perm2}
 \end{subfigure}
 \label{perm0}
 \caption{Permeability coefficients $\kappa_1$ and $\kappa_2$ for numerical implementation.}
\end{figure}
	

\begin{figure}[H]
  \centering
    \begin{subfigure}{0.45\textwidth}
  \includegraphics[width=\textwidth]{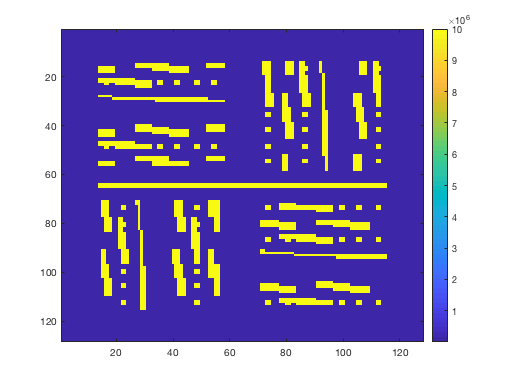}
  \caption{$\hat{Q}(\bfa{x})$. The value in each channel is $10^7$.}
  \label{perm_Q}
\end{subfigure}
\hfill
   \begin{subfigure}{0.45\textwidth}
  \includegraphics[width=\textwidth]{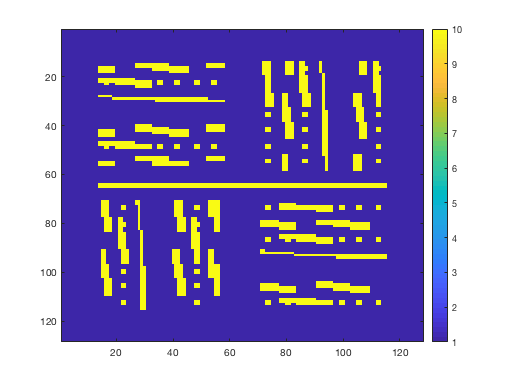}
  \caption{$\tilde{Q}(\bfa{x})$. The value in each channel is $10$.}
  \label{perm_c}
 \end{subfigure}
 \label{perm}
 \caption{Interaction coefficients $Q_1$ and $Q_2$ for numerical implementation.}
\end{figure}
\begin{figure}
  \centering
    \begin{subfigure}{0.45\textwidth}
  \includegraphics[width=\textwidth]{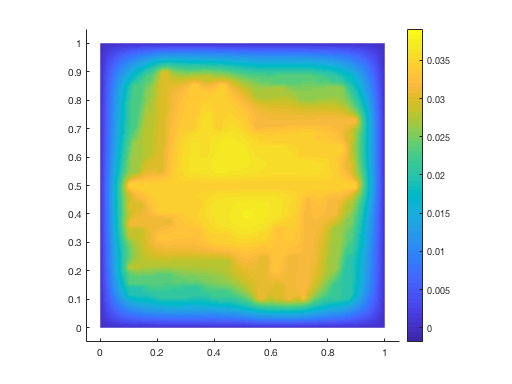}
  \caption{FEM, $u_1$, $DOF_{\tu{fine}} = 32768$.}
  \label{perm_Q2}
\end{subfigure}
  \begin{subfigure}{0.45\linewidth}
 \includegraphics[width=\linewidth]{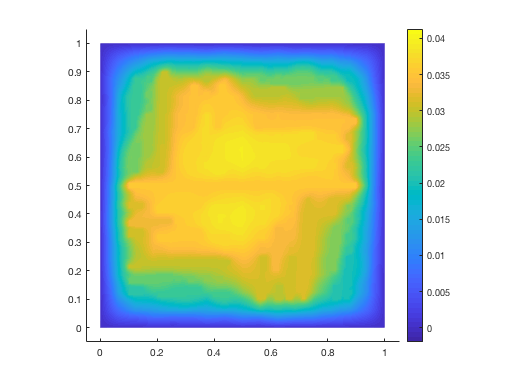}
  \caption{Coupled GMsFEM, $u_1$, $\tu{dim}(\bfa{V}_{\tu{ms}}) = 1350$.}
  \label{perm_c2}
 \end{subfigure}
 \caption{Solutions using FEM and Coupled GMsFEM.}
 \label{Solution_plot}
\end{figure}

\begin{table}
\centering
\begin{tabular}{|c|c|c|c|c|}
  \hline
 \multirow{2}{*}{dim($\bfa{V}_{\tu{ms}}$)} & \multicolumn{2}{|c|}{$u_1$}  & \multicolumn{2}{|c|}{$u_2$} \\ 
 \cline{2-5} & $H^1$ Errors(\%) & $L^2$ Errors(\%) &$H^1$ Errors(\%) & $L^2$ Errors(\%)\\
  \hline \hline
1800& 11.619 & 1.162 & 10.246 & 1.173\\
[.2em]
2700&6.994 &  0.449 & 6.811 & 0.456 \\
[.2em]
3600&6.129 & 0.335& 5.832 & 0.340\\
[.2em]
4500&5.214  & 0.223 &4.768&0.228\\
[.2em]
5400&3.726 & 0.117 &3.532&0.120\\
[.2em]
7200 &2.253 & 0.045&2.186&0.047\\

  \hline
\end{tabular} \\
\caption{ Coupled GMsFEM, $Q_1 =  Q_2 = \hat{Q}$, $DOF_{\tu{fine}}$ = 32768. 
} 
\label{errors1}
\end{table}
\begin{table}
\centering
\begin{tabular}{|c|c|c|c|c|}
  \hline
 \multirow{2}{*}{dim($\bfa{V}_{\tu{ms}}$)} & \multicolumn{2}{|c|}{$u_1$}  & \multicolumn{2}{|c|}{$u_2$} \\ 
 \cline{2-5} & $H^1$ Errors(\%) & $L^2$ Errors(\%) &$H^1$ Errors(\%) & $L^2$ Errors(\%)\\
  \hline \hline
1800& 16.170 & 2.987 & 17.450 & 2.998\\
[.2em]
2700&8.213  & 1.020& 9.976 & 1.026 \\
[.2em]
3600&6.630 & 0.756 & 8.637 & 0.760\\
[.2em]
4500&5.554  & 0.544 &7.490&0.547\\
[.2em]
5400&4.717 & 0.435 &6.776&0.438\\
[.2em]
7200 &2.712 & 0.237&5.065&0.239\\

  \hline
\end{tabular} \\
\caption{ Uncoupled GMsFEM, $ Q_1 =  Q_2 = \hat{Q}$, $DOF_{\tu{fine}}$ = 32768. 
} 
\label{errors2}
\end{table}

\begin{table}
\centering
\begin{tabular}{|c|c|c|c|c|}
  \hline
 \multirow{2}{*}{dim($\bfa{V}_{\tu{ms}}$)} & \multicolumn{2}{|c|}{$u_1$}  & \multicolumn{2}{|c|}{$u_2$} \\ 
 \cline{2-5} & $H^1$ Errors(\%) & $L^2$ Errors(\%) &$H^1$ Errors(\%) & $L^2$ Errors(\%)\\
  \hline \hline
1800& 16.051 & 2.250 & 17.558 & 2.547\\
[.2em]
2700&8.232  & 0.571 & 7.957 & 0.567 \\
[.2em]
3600&6.621 & 0.375& 6.579 & 0.381\\
[.2em]
4500&5.567  & 0.255 &5.374&0.252\\
[.2em]
5400&4.729 & 0.195 &4.578&0.179\\
[.2em]
7200 &2.696 &  0.064&2.628&0.061\\

  \hline
\end{tabular} \\
\caption{ Coupled GMsFEM, $\ Q_1 = -10 \tilde{Q}, Q_2 = - \tilde{Q}$, $DOF_{\tu{fine}}$ = 32768. 
} 
\label{errors3}
\end{table}
\begin{table}
\centering
\begin{tabular}{|c|c|c|c|c|}
  \hline
 \multirow{2}{*}{dim($\bfa{V}_{\tu{ms}}$)} & \multicolumn{2}{|c|}{$u_1$}  & \multicolumn{2}{|c|}{$u_2$} \\ 
 \cline{2-5} & $H^1$ Errors(\%) & $L^2$ Errors(\%) &$H^1$ Errors(\%) & $L^2$ Errors(\%)\\
  \hline \hline
1800& 16.233& 2.314 & 15.873 & 2.266\\
[.2em]
2700&8.213  & 0.581 & 7.951& 0.566 \\
[.2em]
3600&6.620 & 0.377 & 6.54 & 0.381\\
[.2em]
4500&5.563  & 0.258 &5.371&0.252\\
[.2em]
5400& 4.733 & 0.196 & 4.558 &0.180\\
[.2em]
7200 &2.693 & 0.064 &2.626&0.061\\
  \hline
\end{tabular} \\
\caption{ Uncoupled GMsFEM, $Q_1 = -10 \tilde{Q}, Q_2 = - \tilde{Q}$, $DOF_{\tu{fine}}$ = 32768.
} 
\label{errors4}
\end{table}

\section{Conclusions}\label{conclude}

In this paper, we propose a dual-continuum generalized multiscale finite element method (GMsFEM), to speedily and effectively solve a homogenized system of two equations (for fluid flow pressures), with new convection terms and negative interaction coefficients from \cite{rh2}.  These two equations are coupled via some interaction terms, which take into account the flow transports within each continuum and between the dual continua.  Toward this target, we assume that each continuum is globally a system, which is connected to the other throughout the domain and the form of coupling.  Such dual-continuum background can be in any general form
where the above assumptions are relevant.  Within such dual-continuum background, the multiscale flow is simulated by the GMsFEM, which systematically produces either uncoupled or coupled multiscale basis functions (called uncoupled or coupled GMsFEM, respectively).  That is, multiscale basis functions are constructed for the dual-continuum equations, separately for each equation (uncouple GMsFEM), or jointly for the system (coupled GMsFEM).  Our numerical results show that the combination of GMsFEM and dual-continuum approach is able to compute solutions with high efficiency and accuracy, which are even higher when the coupled multiscale basis functions are applied.  In a future contribution, we will extend this strategy to a dual-continuum system of homogenized nonlinear equations.

\bigskip

\vspace{20pt}

\noindent \textbf{Acknowledgements.}  

Mai's work is funded by Vietnam National Foundation for Science and 
Technology Development (NAFOSTED) under grant number 101.99-2019.326.  


\bibliographystyle{plain}
\bibliography{referencesGMsMC}

\end{document}